\documentclass[a4paper,11pt]{amsart}

\usepackage{amsmath}
\usepackage{mathtools}
\usepackage[all]{xy}
\usepackage[latin1]{inputenc}
\usepackage{varioref}
\usepackage{amsfonts}
\usepackage{amssymb}
\usepackage{bbm}
\usepackage{graphicx}
\usepackage{mathrsfs}
\usepackage[hypertexnames=false,backref=page,pdftex,
 	pdfpagemode=UseNone,
 	breaklinks=true,
 	extension=pdf,
 	colorlinks=true,
 	linkcolor=blue,
 	citecolor=red,
 	urlcolor=blue,
 ]{hyperref}

\usepackage{amsmath, multicol,a4}
\usepackage[all]{xy}

\newtheorem{theorem}{\bf Theorem}[subsection]

\newtheorem{prop}[theorem]{\bf Proposition}
\newtheorem{cor}[theorem]{\bf Corollary}
\newtheorem{lemma}[theorem]{\bf Lemma}
\newtheorem{notation}[theorem]{\bf Notation}
\newtheorem{definition}[theorem]{\bf Definition}

\theoremstyle{remark}
\newtheorem{example}[theorem]{\bf Example}
\theoremstyle{remark}
\newtheorem{rem}[theorem]{\bf Remark}

 \numberwithin{equation}{subsection}

\newcommand{\ad}{\operatorname{ad}}
\newcommand{\Ad}{\operatorname{Ad}}

\newcommand{\tr}{\operatorname{tr}}
\newcommand{\Hom}{\operatorname{Hom}}

\def\go{\mathfrak}
\def\bb{\mathbb}
\def\C{\bb C}
\def\Z{\bb Z}
\def\N{\bb N}

\def\cal{\mathcal}

 \def\adots{\mathinner{\mkern2mu\raise1pt\hbox{.}
\mkern3mu\raise4pt\hbox{.}\mkern1mu\raise7pt\hbox{.}}}

 \title[Graded Lie algebras]{Minimal Graded Lie algebras and representations of  quadratic algebras}

\author{Hubert Rubenthaler}

\address
{Hubert Rubenthaler\\ Institut de Recherche Math\'ematique Avanc\'ee\\
Universit\'e de Strasbourg et CNRS\\
7 rue Ren\'e Descartes\\
67084 Strasbourg Cedex\\ France\\
E-mail: {\tt rubenth@math.unistra.fr}}

\begin{document}
\parindent=0pt
 
 \maketitle

\begin{abstract} {\bf Let $({\go g}_{0},B_{0})$ be a quadratic Lie algebra (i.e. a Lie algebra $\go{g}_{0}$ with a   non degenerate symmetric invariant bilinear form $B_{0}$) and let $(\rho,V)$ be a finite dimensional representation of ${\go g}_{0}$. We define on $ \Gamma(\go{g}_{0}, B_{0}, V)=V^*\oplus  {\go g}_{0}\oplus V$ a structure of local Lie algebra in the sense of Kac (\cite{Kac1}), where the bracket between $\go{g}_{0}$ and $V$ (resp. $V^*)$ is given by the representation $\rho$ (resp. $\rho^*$), and where the bracket between $V$ and $V^*$ depends on $B_{0}$ and $\rho$. This implies the existence of two $\Z$-graded Lie algebras ${\go g}_{max}(\Gamma(\go{g}_{0}, B_{0}, V))$ and ${\go g}_{min}(\Gamma(\go{g}_{0}, B_{0}, V))$ whose local part is  $\Gamma(\go{g}_{0},B_{0}, V)$. We investigate these graded Lie algebras, more specifically in the case where ${\go g}_{0}$ is reductive. Roughly speaking, the map $(\go{g}_{0},B_{0}, V)\longmapsto {\go g}_{min}(\Gamma(\go{g}_{0}, B_{0}, V))$ a bijection between triplets   and a class of graded Lie algebras.  We show that the existence of  "associated $\go {sl}_{2}$-triples" is equivalent to the existence of non trivial relative invariants on some orbit, and we define the "graded Lie algebras of polynomial type" which give rise to some  dual airs.} 

 \end{abstract}
 
\maketitle
\vfill\eject
 {\small \def\contentsname{Table of Contents}
\tableofcontents}
\medskip\medskip\medskip\medskip\medskip\medskip
 
 \section{Introduction}
 
 In this paper,  following \cite{Kac1}, a graded Lie algebra is a Lie algebra  ${\go g}= \oplus_{i\in \Z}\,{\go g}_{i}$, such that 
      $\dim {\go g}_{i}<+\infty$,  such that $[{\go g}_{i},{\go g}_{j}]\subset {\go g}_{i+j}$, for all $i,j \in \Z$, and such that ${\go g}$ is generated by its {\it local part}  $ {\go g}_{-1}\oplus {\go g}_{0}\oplus {\go g}_{1}$.

 If  $\go{g}=\oplus_{i=-n}^{n} \go{g}_{i}$ is a grading of a (finite dimensional) complex semi-simple Lie algebra,  it is well known that if $B$ denotes the Killing form of $\go{g}$, then $B(\go{g}_{i},\go{g}_{j})=0 $ if $i+j \neq 0$.  This allows us to identify $\go{g}_{-1}$ with the dual $\go{g}_{1}^*$. Moreover as $B$ is invariant  the  bracket representation $(\go{g}_{0}, \go{g}_{-1})$  can be identified with the dual representation $(\go{g}_{0},\go{g}_{1}^*)$.
 
 It is then a natural question to ask if any finite dimensional representation $(\go{g}_{0},\rho,  V)$ of a finite dimensional Lie algebra $\go{g}_{0}$ can be embedded in a  graded Lie algebra $\go{g}=\oplus_{i=-\infty}^{+\infty} \go{g}_{i}$ such that $(\go{g}_{0}, \go{g}_{1})\simeq (\go{g}_{0},\rho, V)$ and $(\go{g}_{0}, \go{g}_{-1})\simeq (\go{g}_{0},\rho^*, V^*)$, and such that the bracket between $V$ and $V^*$ is non trivial. 
 
 The first result of this paper is to give 	a positive answer  to this question for any representation of a quadratic Lie algebra. A quadratic Lie algebra is a pair $(\go{g}_{0}, B_{0})$ where  $\go{g}_{0}$ is a Lie algebra and $B_{0}$ a non-degenerate  invariant symmetric bilinear form on $\go{g}_{0}$.  Of course the definition of the bracket between $V$ and $V^*$ will depend  on $B_{0}$ and $\rho$.

 We will use a result of V. Kac  (\cite{Kac1}) which asserts that in order to construct a graded Lie algebra  $\go{g}=\oplus_{i\in \Z}  \go{g}_{i}$ it suffices to construct the local part $\Gamma=\go{g}_{-1}\oplus \go{g}_{0}\oplus \go{g}_{1}$, which has to be endowed with a partial Lie bracket (see section 2 for details). Therefore once we have build   the partial bracket on the  local part  $\Gamma({\go g}_{0},B_0,\rho)=V^*\oplus \go{g}_{0}\oplus V$,  the existence of the "global" Lie algebra is just an application of a  result of Kac. In fact Kac theory provides us with  two such graded Lie algebras: a maximal one (denoted here $\go{g}_{max}(\Gamma({\go g}_{0},B_0,\rho)) $) and a minimal one (denoted $\go{g}_{min}(\Gamma({\go g}_{0},B_0,\rho)))$. Any graded Lie algebra with a given local part is a quotient of the maximal algebra, and has a quotient isomorphic to the minimal one. Of course, in general, these algebras are infinite dimensional.
 
 \vskip 5pt

Lets us now    give a more precise description of the paper.

 \vskip 5pt

In section $2.1$ we give a brief account of the results of Kac that we will use. 

  In section $2.2$ we prove    a general result concerning this  construction.   Let  $\Gamma= \go{g}_{-1}\oplus \go{g}_{0}\oplus \go{g}_{1}$ be  a local Lie algebra and let  $ \go{g}_{min}(\Gamma)= \oplus_{-\infty}^{+\infty}\go{g}_{i} $ be  the minimal graded Lie algebra whith local part $\Gamma$. Let $|n|\geq 2$. We show that there exists a universal polynomial ${\cal P}_{n}$ defined on the local part $\Gamma$ such that $\go{g}_{n}=\{0\}$ if and only if the identity ${\cal P}_{n}=0$ is satisfied on the local Lie algebra $\Gamma$. 
  
   \vskip 5pt

In section 3.1  we  define  the local Lie algebra structure on $\Gamma({\go g}_{0},B_0,\rho)=V^*\oplus \go{g}_{0}\oplus V$ (see   Theorem \ref{th-algebre-locale}) . In   section 3.2 we give necessary and sufficient conditions  for the algebras $\Gamma(\go{g}_{0}^1, B_{0}^1,\rho_{1})$ and $\Gamma(\go{g}_{0}^2, B_{0}^2,\rho_{2})$ corresponding to two fundamental data to be isomorphic. We also investigate  the dependence of the local structure on $B_{0}$ and $\rho$ and show that any isomorphism between  the "fundamental triplets" $(\go{g}_{0}^1, B_{0}^1, (\rho_{1},V_{1}))$ and $(\go{g}_{0}^2, B_{0}^2, (\rho_{1},V_{2}))$  can be extended to an isomorphism between $\Gamma(\go{g}_{0}^1, B_{0}^1,\rho_{1})$ and $\Gamma(\go{g}_{0}^2, B_{0}^2,\rho_{2})$. 
In section 3.3 we apply Kac Theorem  to obtain  the minimal and maximal Lie algebras associated to the local Lie algebra  $\Gamma({\go g}_{0},B_0,\rho)=V^*\oplus \go{g}_{0}\oplus V$. We also prove, that under some conditions, the  reductive graded Lie algebras are always minimal graded Lie algebras (Proposition  \ref{prop-reductive=minimal}). 
 Section 3.4 deals with another important notion  for graded Lie algebras due to Kac, namely the transitivity (see Definition \ref{def-transitivity}). We give a necessary and sufficient condition for the local Lie algebra $\Gamma({\go g}_{0},B_0,\rho)$, or the minimal Lie algebra $\go{g}_{min}(\Gamma({\go g}_{0},B_0,\rho))$,  to be transitive (Proposition \ref{prop-transitive}).  
 We also prove that if $\go{g}_{0}$ is reductive, then under some conditions including the transitivity of  $\Gamma({\go g}_{0},B_0,\rho)$,  the fact that $\go{g}_{min}({\Gamma({\go g}_{0},B_0,\rho))}$ is finite dimensional  implies   that $\go{g}_{min}({\Gamma({\go g}_{0},B_0,\rho))}$ is semi-simple (see Proposition \ref{prop-finie-ss}).
 In section 3.5 we show  that the form $B_{0}$ extends uniquely to an invariant symmetric bilinear form $B$ on $\go{g}_{min}(\Gamma({\go g}_{0},B_0,\rho))$. Moreover if  the local part is transitive then  the form $B$ is nondegenerate (Proposition \ref{prop-forme-globale}). This allows us to show that there exists a bijection between some equivalence classes of fundamental triplets and the equivalence classes of transitive graded Lie algebras endowed with a non-degenerate symmetric bilinear form $B$ such that $B(\go{g}_{i},\go{g}_{j})=0 $ if $i\neq-j$ (Theorem \ref{th-bijection}). 
 
  \vskip 5pt
In section 4.1 we give a necessary and sufficient condition  for the existence of an $\go{sl}_{2}$-triple $(Y,H_{0},X)$ where $Y\in V^*$, $X\in V$ and where $H_{0}$ is the "grading element" of the center of $\go{g}_{0}$ defined by  the condition $\rho(H_{0})_{|_{V}}=2 \text{Id}_{V}$ (see Theorem \ref{th-CNS-sl2}). In section 4.2 we  assume that the representation $\rho$ lifts to a representation of a connected algebraic group $G_{0}$ with Lie algebra $\go{g}_{0}$. We prove then that the existence of such an $\go{sl}_{2}$-triple is also equivalent  to the existence of a $G_{0}$-orbit in $V$ supporting a non trivial rational relative invariant  (see Theorem \ref{th-equivalence-invariants-(P)-sl2}). Moreover, if ${\cal O}_{X}$ is such an orbit, and if $x\in {\cal O}_{X}$, the map $x\longmapsto \varphi(x)$, where $\varphi$ is the set of $y\in \go{g}_{-1}$ such that $(y,H_{0},x)$ is an $\go{sl}_{2}$-triple, can be viewed as an equivariant section of the cotangent bundle of the orbit of $x$ (Proposition \ref{prop-y=section-cotangent}).

\vskip 5pt
  
  Section 5 is devoted to the so-called graded Lie algebras of polynomial type. These algebras are defined in section 5.1 to be the minimal Lie algebras associated to $\go{g}_{0}=\go{gl}(W)$ and to the "natural " representation  of $\go{sl}(W)=\go{g}_{0}'$ on the space $V=\C^p[W]$ of homogeneous  polynomials of degree $p$ on $W$ (see Definition \ref{def-type-symplectique}). We classify the finite dimensional Lie algebras of polynomial type (Proposition \ref{prop-symplectic-finite}) and show that these algebras are always associated to $\go{sl}_{2}$-triples in the sense of section 4 (Theorem \ref{th-symplectique-sl2}). Finally, in section 5.2 we show, under some assumptions, that  if $(A,W)$ is an irregular reductive regular prehomogeneous vector space, then the semi-simple part $\go{a}'$ of the Lie algebra of $A$ is the member of a dual pair in the  Lie algebras of polynomial type associated to $\C^p[W]$, where $p$ is the degree of the fundamental relative invariant of $(A,W)$ (Theorem \ref{th-paire-duale}).
  
  \vskip 10pt
  {\bf Acknowledgment:}  
  
   I would  like tho thank the referee for his careful reading of the manuscript and for his numerous remarks, suggestions and corrections which have allowed me to improve the original text.
   
   I would like to thank Michel Brion who communicated me Lemma \ref{lem-P}.  
 
 
 
\vskip 20pt
 \section{Graded Lie algebras and local Lie algebras}
 \vskip 10pt

 In this paper all the algebras are defined over the field $\C$ of complex numbers. 
 \subsection{Maximal and minimal algebras}\hfill
\vskip 5pt
Recall from the Introduction that  by a graded Lie algebra we mean  a $\Z$-graded Lie algebra ${\go g}= \oplus_{i\in \Z}\,{\go g}_{i}$, such that  

a) $\dim {\go g}_{i}<+\infty$ ,  $[{\go g}_{i},{\go g}_{j}]\subset {\go g}_{i+j}$, for all $i,j \in \Z$, 

b) $\go g$ is generated by  ${\go g}_{-1}\oplus {\go g}_{0}\oplus {\go g}_{1}$.
 
  
   
   
 
 If ${\go g}$ is a graded Lie algebra, the subspace $\Gamma( {\go g})={\go g}_{-1}\oplus {\go g}_{0}\oplus {\go g}_{1}$ is called the {\it local part } of ${\go g} $.
 
 \begin{definition} \hfill
 
    {\rm1)} A local Lie algebra is a direct sum $\Gamma=\go{g}_{-1}\oplus\go{g}_{0}\oplus \go{g}_{1}$ of finite dimensional subspaces such that if $|i+j|\leq 1$ there exists a bilinear anticommutative operation $\go{g}_{i}\times \go{g}_{j}\rightarrow \go{g}_{i+j} ((x,y)\rightarrow [x,y])$ such that the Jacobi identity $[x,[y,z]]=[[x,y],z]+[y,[x,z]]$ holds each time the three terms of the identity are defined.

       {\rm2)} A symmetric bilinear form $B_{\Gamma}$ on a local Lie algebra $\Gamma$ is said to be invariant if the identity
       $$B_{\Gamma}([x,y],z)= B_{\Gamma}(x,[y,z])$$
       holds for $x,y,z\in \Gamma$ each time that the brackets are defined.
 \end{definition}
 
 \begin{definition}\hfil
 
 $a)$ Let $\go{g}^1=\oplus_{n\in \Z} \go{g}^1_{n}$ and $\go{g}^2=\oplus_{n\in \Z} \go{g}^2_{n}$ be two graded Lie algebras. A homomorphism of graded Lie algebras from  $\go{g}^{1}$ to $\go{g}^{2}$ is a map $\Psi: \go{g}^{1}\longrightarrow \go{g}^{2}$ which is a homomorphism of Lie algebras such that $\forall n\in \N, \Psi(\go{g}^1_{n})\subset \go{g}^2_{n}$.
  
  $b)$ Let $\Gamma_{1}=\go{g}_{-1}^1\oplus\go{g}_{0}^1\oplus \go{g}_{1}^1$ and $\Gamma_{2}=\go{g}_{-1}^2\oplus\go{g}_{0}^2\oplus \go{g}_{1}^2$ be two local Lie algebras. A homomorphism of local Lie algebras between $\Gamma_{1}$ and $\Gamma_{2}$ is a linear map $\Psi:\Gamma_{1}\longrightarrow \Gamma_{2}$ such that $\forall n\in\{-1,0,1\}, \Psi(\go{g}^1_{n})\subset \go{g}^2_{n}$, and such that $\Psi([x,y])=[\Psi(x),\Psi(y)]$, for all $x,y\in \Gamma_{1}$ such that the bracket $[x,y]$ is defined.
 \end{definition}
 
 Of course the local part $\Gamma({\go g})$ of a graded Lie algebra ${\go g}$, endowed with the bracket of ${\go g}$ is a local Lie algebra. A natural question is to know if, for a given  local Lie algebra  $\Gamma$, there exists a graded Lie algebra whose local part is $\Gamma$. The answer is "yes". More precisely we have:
 \begin{theorem}\label{th-Kac}{\rm(Kac, \cite{Kac1}, Proposition 4) }\hfill
 
 Let $\Gamma=\go{g}_{-1}\oplus \go{g}_{0}\oplus \go{g}_{1}$ be a local Lie algebra. 
 
 {\rm 1)} There exists a unique graded Lie algebra  $\go{g}_{max}(\Gamma)$ whose local part is $\Gamma$ and which satisfies the following universal property.
 
 Any morphism of local Lie algebras $\Gamma\rightarrow \Gamma({\go g})$ from $\Gamma$ into the local part $\Gamma({\go g})$ of a graded Lie algebra ${\go g}$ extends uniquely to a morphism of graded Lie algebras $\go{g}_{max}(\Gamma)\rightarrow {\go g}$. $($And hence any graded Lie algebra whose local part is isomorphic to $\Gamma$, is a quotient of $\go{g}_{max}(\Gamma)$$)$.
 Moreover we have 
 $$\go{g}_{max}(\Gamma)=F(\go{g}_{-1})\oplus \go{g}_{0}\oplus F(\go{g}_{1}),$$
 where $F(\go{g}_{-1})$ $($resp. $F(\go{g}_{1})$$)$ is the free Lie algebra generated by $\go{g}_{-1}$ $($resp. $\go{g}_{1}$$)$.
 
 {\rm 2)} There exists a unique graded Lie algebra  $\go{g}_{min}(\Gamma)$ whose local part is $\Gamma$ and which satisfies the following universal property.
 
 Any surjective morphism of local Lie algebras $\Gamma(\go{g})\rightarrow  \Gamma$ from the local part of a graded Lie algebra ${\go g}$ into $\Gamma$ extends uniquely to a $($surjective$)$ morphism of graded Lie algebras $ {\go g}\rightarrow \go{g}_{min}(\Gamma)$. $($And hence $\go{g}_{min}(\Gamma)$ is a quotient of any graded Lie algebra whose local part is isomorphic to $\Gamma$$)$.
 
 In fact $\go{g}_{max}(\Gamma)$ has a unique maximal graded ideal $J_{max}$ such that $J_{max}\cap \Gamma= \{0\}$, and $\go{g}_{max}(\Gamma)/J_{max}=\go{g}_{min}(\Gamma)$.
\end{theorem}

\begin{prop}\label{prop-Kac}{\rm(Kac, \cite{Kac1}, Proposition 7) }\hfill

Let ${\go g}= \oplus_{i\in \Z}\,{\go g}_{i}$. Let $B$ be a bilinear invariant symmetric form on the local part $\Gamma(\go g)$, such that $B(\go{g}_{i}, \go{g}_{j})=0$ when $i+j\neq 0$ $(i,j \in \{-1,0,1\})$. Then $B$ extends uniquely to a bilinear invariant symmetric form on $\go{g}$ $($still denoted $B$$)$ such that $B(\go{g}_{i}, \go{g}_{j})=0$ when $i+j\neq 0$ $(i,j \in \Z)$.

\end{prop}
\vskip5pt
\subsection{ When is $\dim(\go{g}_{min}(\Gamma))< +\infty?$}  \hfill
\vskip 5pt
  Let $\Gamma$ be a local Lie algebra. Consider the minimal graded Lie algebra  $\go{g}_{min}(\Gamma)$ associated to $\Gamma$. A natural question is to ask how one can see, only from the knowledge of the  local Lie algebra $\Gamma$, whether or not $\go{g}_{min}(\Gamma)$ is finite dimensional. For example suppose that $\Gamma$ is already a  Lie algebra, but we do not know anything about the brackets between two elements of $\go{g}_{1}$ or of $\go{g}_{-1}$.   Then    we have surely 
  $${\cal P}_{2}(Y,X_{1},X_{2})= [[Y,X_{1}],X_{2}]+ [X_{1},[Y, X_{2}]]=0 \eqno(2-2-1)$$ for all $Y\in \go{g}_{-1} $ and all $X_{1},X_{2}\in \go{g}_{-1}$. This is because this element is equal to 
  $$[Y,[X_{1},X_{2}]] =0\eqno(2-2-2)$$
   in $\go{g}_{min}(\Gamma)$. The first equation makes sense in $\Gamma$, but not the second. Conversely, if the relation $(2-2-1)$  holds in the local Lie algebra $\Gamma$, then $\Gamma=\go{g}_{min}(\Gamma)$ is in fact a $3$-graded Lie algebra.
  
  We will show that there exists a "universal" polynomial identity ${\cal P}_{n}=0$ in any  local Lie algebra $\Gamma$ which is a necessary and sufficient condition for having $\go{g}_{n}=\{0\}$ (see Theorem \ref{th-finitude-generale} below).
  
  Let us denote by ${\cal V}_{n}= \{Y_{1},\dots,Y_{n-1}, X_{1},\dots,X_{n}\}$ a set of $2n-1$ variables. Let $F({\cal V}_{n})$ be the free Lie algebra on ${\cal V}_{n}$.

  A {\it   Lie monomial of degree one} in the variables  ${\cal V}_{n}$ is an element of ${\cal V}_{n}$. By induction   a {\it Lie monomial of degree $k$} in the variables  ${\cal V}_{n}$  is an element of  $F({\cal V}_{n})$ of the form $[u,v]$ where $u$ is a Lie monomial of degree $i \,\,(1\leq i< k)$ and $v$ is a Lie monomial of degree $k-i$ . A {\it Lie polynomial } ${\cal P}$  in the variables  ${\cal V}_{n}$ is a linear combination of monomials (in other words it is just an element of $F({\cal V}_{n})$).

  \begin{prop}\label{prop-poly-P_{n}}\hfill
  
  Let $n\geq 2$. The Lie monomial  $[Y_{1},[Y_{2},[\dots,[Y_{n-1},[X_{1},[\dots,[X_{n-1},X_{n}]\dots]$  is equal $($as elements of the free algebra  $F({\cal V}_{n})$$)$ to  a Lie polynomial of the form 
  $${\cal P}_{n}(Y_{1},\dots,Y_{n-1},X_{1}, \dots,X_{n})=\sum_{\alpha}U_{\alpha}(Y_{1},\dots,Y_{n-1},X_{1}, \dots,X_{n})$$
  where each monomial $U_{\alpha}$ makes sense in the local Lie algebra $\Gamma$ when  $Y_{i}\in \go{g}_{-1}, i=1,\dots,n-1$ and when $X_{j}\in \go{g}_{1}, j=1,\dots,n$.

  \end{prop}
  
  \begin{proof} \hfill
  
  As we have already noticed we have 
$$[Y_{1},[X_{1},X_{2}]]={\cal P}_{2}(Y_{1},X_{1},X_{2}]]=[[Y_{1},X_{1}],X_{2}]]+[X_{1},[Y_{1},X_{2}]]$$
and the right hand side is defined in the local Lie algebra if $Y_{1}\in \go{g}_{-1}$, and $X_{1},X_{2}\in \go{g}_{1}$.

Suppose now that the result is true for $n-1$. We start with the monomial
$$[Y_{1},[Y_{2},[\dots,[Y_{n-1},[X_{1},[\dots,[X_{n-1},X_{n}]\dots]\eqno (2-2-3)$$
Consider also the sub-monomial
$$[Y_{n-1},[X_{1},[\dots,[X_{n-1},X_{n}]\dots] $$
Using the Jacobi identity we get 
$$\begin{array}{l}[Y_{n-1},[X_{1},[\dots,[X_{n-1},X_{n}]\dots]\\
= \sum_{i=1}^n[X_{1},[\dots,[[Y_{n-1},X_{i}],[X_{i+1},[\dots,[X_{n-1},X_{n}]\dots]
\end{array}$$
Consider again the sub-monomial $[[Y_{n-1},X_{i}],[X_{i+1},[\dots,[X_{n-1},X_{n}]\dots]$, and set $U_{i}=[Y_{n-1},X_{i}]$ (notice that if $Y_{n-1}\in \go{g}_{-1}$, and $X_{i}\in \go{g}_{1}$, then $U_{i}\in \go{g}_{0}$). From the Jacobi identity we obtain
$$\begin{array}{l}
[[Y_{n-1},X_{i}],[X_{i+1},[\dots,[X_{n-1},X_{n}]\dots]\\
=[U_{i},[X_{i+1},[\dots,[X_{n-1},X_{n}]\dots]\\
= \sum_{k=i+1}^n[X_{i+1},[\dots,[[U_{i},X_{k}],[\dots,[X_{n-1},X_{n}]\dots]\\
= \sum_{k=i+1}^n[X_{i+1},[\dots,[\widetilde{X_{k}^{i}},[\dots,[X_{n-1},X_{n}]\dots]
\end{array}
$$
where $\widetilde{X_{k}^{i}}=[U_{i},X_{k}]=[[Y_{n-1},X_{i}],X_{k}]$ is again a monomial which makes sense in $\Gamma$.
Hence we have obtained 
$$\begin{array}{l}
[Y_{n-1},[X_{1},[\dots,[X_{n-1},X_{n}]\dots]\\
=\sum_{i=1}^n\sum_{k=i+1}^n [X_{1},[X_{2},[\dots,[X_{i-1},[X_{i+1},[\dots,[\widetilde{X_{k}^{i}},[\dots,[X_{n-1},X_{n}]\dots]
\end{array}$$
and finally   we get:

$$\begin{array}{l}
[Y_{1},[Y_{2},[\dots,[Y_{n-1},[X_{1},[\dots,[X_{n-1},X_{n}]\dots]\\
=\displaystyle\sum_{i=1}^n\sum_{k=i+1}^n [Y_{1},[Y_{2},\dots[Y_{n-2}, [X_{1},\dots[X_{i-1},[X_{i+1},[\dots[\widetilde{X_{k}^{i}},[\dots[X_{n-1},X_{n}]\dots]

\end{array}$$
But this last expression is a sum of monomials of type $(2-2-3)$  with one $"Y"$ and one $"X"$ less and therefore 
$$\begin{array}{l}
[Y_{1},[Y_{2},[\dots,[Y_{n-1},[X_{1},[\dots,[X_{n-1},X_{n}]\dots]\\
=\displaystyle\sum_{i=1}^n\sum_{k=i+1}^n {\cal P}_{n-1}(Y_{1},\dots,Y_{n-2},X_{1},\dots,X_{i-1},X_{i+1},\dots,\widetilde{X_{k}^{i}},\dots,X_{n}).
\end{array}$$
 As $\widetilde{X_{k}^{i}}= [[Y_{n-1},X_{i}],X_{k}]$ we obtain by  induction that each 
 
 ${\cal P}_{n-1}(Y_{1},\dots,Y_{n-2},X_{1},\dots,X_{i-1},X_{i+1},\dots,\widetilde{X_{k}^{i}},\dots,X_{n})$ is a sum of monomials which make sense in $\Gamma$ if $X_{i}\in \go{g}_{1}$ and $Y_{j}\in \go{g}_{-1}$.

\end{proof}

Denote by $\go{g}_{max}(\Gamma)=\oplus _{i\in \Z}(\go{g}_{max}(\Gamma))_{i}$ the grading in $\go{g}_{max}(\Gamma)$. Remember that $(\go{g}_{max}(\Gamma))_{i}=\go{g}_{i}$ for $i=-1,0,1$

\begin{lemma} \label{lemme-ideal}
 Let $n\geq 2$. The polynomial identity ${\cal P}_{n}(Y_{1},\dots,Y_{n-1},X_{1}, \dots,X_{n})=0$ is satisfied in $\Gamma$ for $X_{i}\in \go{g}_{1}$ and $Y_{j}\in \go{g}_{-1}$ if and only if  the vector space
 $$J_{n}^+= \sum_{k=0}^{n-2}(\ad \go{g}_{-1})^k(\oplus_{i\geq n}(\go{g}_{max}(\Gamma))_{i})$$
 is a graded ideal of $\go{g}_{max}(\Gamma)$, contained in $\oplus_{i\geq 2} (\go{g}_{max}(\Gamma))_{i}$
\end{lemma}

\begin{proof} \hfill

Suppose that the identity ${\cal P}_{n}(Y_{1},\dots,Y_{n-1},X_{1}, \dots,X_{n})=0$ is satisfied in $\Gamma$.
It is clear from the definition that $J_{n}^+$ is stable under $\ad( \oplus_{i=0}^\infty (\go{g}_{max}(\Gamma))_{i}$. It remains to prove that $\ad{Y}(J_{n}^+)\subset J_{n}^+$ for $Y\in \go{g}_{-1}$. We have 
$$[Y, (\ad \go{g}_{-1})^k(\oplus_{i\geq n}(\go{g}_{max}(\Gamma))_{i})]\subset (\ad \go{g}_{-1})^{k+1}(\oplus_{i\geq n}(\go{g}_{max}(\Gamma))_{i}).$$
Therefore it is enough to show that $$(\ad \go{g}_{-1})^{n-1}(\go{g}_{max}(\Gamma))_{n}=\{0\} \eqno (*)$$
From the definition of a graded Lie algebra, $\go{g}_{max}(\Gamma)$ is generated by its local part $\Gamma$, and this implies that $(\go{g}_{max}(\Gamma))_{n}$ is the space of linear combinations of Lie monomials of the form $[X_{1},[X_{2},\dots[X_{n-1},X_{n}]\dots]$ with $X_{i}\in\go{g}_{1}$. As the identity ${\cal P}_{n}(Y_{1},\dots,Y_{n-1},X_{1}, \dots,X_{n})= [Y_{1},[Y_{2},[\dots[Y_{n-1},[X_{1},[X_{2},[\dots[X_{n-1},X_{n}]\dots]$ holds in $\go{g}_{max}(\Gamma)$, we obtain condition $(*)$.  

Conversely suppose that $J_{n}^+$ is a graded ideal of $\go{g}_{max}(\Gamma)$. Then the identity $(*)$ above holds. This implies that the identity  ${\cal P}_{n}(Y_{1},\dots,Y_{n-1},X_{1}, \dots,X_{n})=0$ is true.

\end{proof}

\begin{rem}\label{rem-finitude-symetrie}\hfill

Of course, symmetrically, if the identity ${\cal P}_{n}(X_{1},\dots,X_{n-1},Y_{1},\dots,Y_{n})=0$ is verified for any  elements $X_{i}\in \go{g}_{1}$ and $Y_{j}\in \go{g}_{-1}$, then  $$J_{n}^-= \sum_{k=0}^{n-2}(\ad \go{g}_{1})^k(\oplus_{i \leq -n}(\go{g}_{max}(\Gamma))_{i})$$
is a graded ideal of  $\go{g}_{max}(\Gamma)$, contained in $\oplus_{i\leq -2} (\go{g}_{max}(\Gamma))_{i}$

\end{rem}

\begin{theorem}\label{th-finitude-generale}\hfill

Let $\Gamma$ be a local Lie algebra. Denote by $\go{g}_{min}(\Gamma)=\oplus _{i\in \Z}(\go{g}_{min}(\Gamma))_{i}$ the grading in $\go{g}_{min}(\Gamma)$.

$1)$ 
The following conditions are equivalent:

$a)$ $(\go{g}_{min}(\Gamma))_{i}=\{0\}$ for $i\geq n$

$b)$ The identity ${\cal P}_{n}(Y_{1},\dots,Y_{n-1},X_{1}, \dots,X_{n})=0$ holds in $\Gamma$, for all $X_{i}\in \go{g}_{1}$ and all $Y_{j}\in \go{g}_{-1}$.

$2)$ 
The following conditions are equivalent:

$a)$ $(\go{g}_{min}(\Gamma))_{i}=\{0\}$ for $i\leq -n$

$b)$ The identity ${\cal P}_{n}(X_{1},\dots,X_{n-1},Y_{1}, \dots,Y_{n})=0$ holds in $\Gamma$, for all $X_{i}\in \go{g}_{1}$ and all $Y_{j}\in  \go{g}_{-1}$.

$3)$ $\dim \go{g}_{min}(\Gamma) < +\infty$ if and only there exist $m, n\in \Z$ such that the identities  ${\cal P}_{m}(Y_{1},\dots,Y_{n-1},X_{1}, \dots,X_{n})=0$ and ${\cal P}_{n}(X_{1},\dots,X_{n-1},Y_{1}, \dots,Y_{n})=0$ hold in $\Gamma$, for all $X_{i}\in  \go{g}_{1}$ and all $Y_{j}\in \go{g}_{-1}$.

\end{theorem}

\begin{proof}\hfill

Lets us prove $1)$. 

$a) \Longrightarrow b)$: if $(\go{g}_{min}(\Gamma))_{n}=\{0\}$, then $[X_{1},[\dots,[X_{n-1},X_{n}]\dots]]=0$. And hence ${\cal P}_{n}(Y_{1},\dots,Y_{n-1},X_{1}, \dots,X_{n})=[Y_{1},[Y_{2},[\dots,[Y_{n-1},[X_{1},[\dots,[X_{n-1},X_{n}]\dots]]=0$.  

$b) \Longrightarrow a)$: suppose that the identity ${\cal P}_{n}(Y_{1},\dots,Y_{n-1},X_{1}, \dots,X_{n})=0$ holds in $\Gamma$.     Let $J_{n}^+= \sum_{k=0}^{n-2}(\ad \go{g}_{-1})^k(\oplus_{i\geq n}\go{g}_{max}(\Gamma))_{i})$ the ideal that we considered in Lemma \ref{lemme-ideal}. Obviously $J_{n}^+\cap \Gamma =\{0\}$ and $\oplus_{i\geq n}(\go{g}_{max}(\Gamma))_{i}\subset J_{n}^+$

Let $J_{max}$ be the maximal graded ideal of $\go{g}_{max}(\Gamma)$ which intersects $\Gamma$ trivially (see Theorem \ref{th-Kac}).  As $\oplus_{i\geq n}(\go{g}_{max}(\Gamma))_{i}\subset J_{n}^+\subset J_{max}$, and as $\go{g}_{min}(\Gamma)=\go{g}_{max}(\Gamma)/J_{max}$, we obtain that $(\go{g}_{min}(\Gamma))_{i}=\{0\}$ for $i\geq n$.

The proof of $2)$ is similar, using the ideal $J_{n}^-$ introduced in Remark \ref{rem-finitude-symetrie}.

Assertion $3)$ is an immediate consequence of $1)$ and $2)$.

\end{proof}

\begin{example}\label{ex-P3}
A direct computation  shows that 
$$\begin{array}{rl}
{\cal P}_{3}(Y_{1},Y_{2},X_{1},X_{2},X_{3})&=[[Y_{1},[[Y_{2},X_{1}],X_{2}]],X_{3}]+[[[Y_{2},X_{1}],X_{2}],[Y_{1},X_{3}]]\\
{}&+[[Y_{1},X_{2}],[[Y_{2},X_{1}],X_{3}]]+[X_{2},[Y_{1},[[Y_{2},X_{2}],X_{3}]]]\\
{}&+[[Y_{1},X_{1}], [[Y_{2},X_{2}],X_{3}]]+[X_{1},[Y_{1},[[Y_{2},X_{2}],X_{3}]]]\\
{}&+[[Y_{1},X_{1}],[X_{2},[Y_{2},X_{3}]]]+[X_{1},[Y_{1},[X_{2},[Y_{2},X_{3}]]]].
\end{array}$$
\end{example}

   \begin{cor} \label{cor-finitude-special} \hfill
   
    Let $\Gamma$ be a local Lie algebra. Suppose that $\go{g}_{min}(\Gamma)=\oplus_{i\in \Z}$ is endowed with a non-degenerate bilinear symmetric invariant form $B$ such that such that $B(\go{g}_{i}, \go{g}_{j})=0$ when $i+j\neq 0$. Then $\dim \go{g}_{i} =\dim \go{g}_{-i} $ for all $i\in \Z$ and $\dim  \go{g}_{min}(\Gamma)< +\infty$ if and only if there exists an integer $n\geq2$ such that the identity ${\cal P}_{n}(Y_{1},\dots,Y_{n-1},X_{1},\dots,X_{n})$ holds in the local Lie algebra $\Gamma $, where ${\cal P}_{n}$ is the Lie polynomial defined in Proposition \ref{prop-poly-P_{n}}. More precisely in that case $\go{g}_{i}=\{0\}$ for $|i|\geq n$.

  \end{cor}

\vskip 20pt

\section{Local and graded Lie algebras associated to $(\go{g}_{0},B_{0},\rho)$}

\vskip 5pt
\subsection {The local Lie algebra $\Gamma({\go g}_{0},B_0,\rho)$} \hfill
 \vskip5pt
 A {\it quadratic Lie algebra} is a pair $(\go{g}_{0},B_{0})$ where $\go{g}_{0}$ is finite dimensional Lie algebra and where $B_{0}$ is  an  invariant nondegenerate symmetric bilinear form on $\go{g}_{0}$.  The most obvious examples of such algebras are the semi-simple algebras (endowed with the Killing form), the commutative algebras (endowed with any symmetric nondegenerate bilinear form) or more generally the reductive Lie algebras. But there exist more sophisticated examples. In general a quadratic Lie algebra, which is  not a direct sum of two orthogonal ideals,  is shown to be  obtained by a finite number  of so-called {\it double extensions} by either a simple  Lie algebra or a one dimensional Lie algebra.  For details see \cite{Medina-Revoy}.
 
 Let $({\go g}_{0},B_{0})$ be a quadratic Lie algebra. Let $(\rho,V)$ be a finite dimensional representation of ${\go g}_{0}$. Let $(\rho^*,V^*)$ be the contragredient representation. We   will often just denote these modules by  $({\go g}_{0},V)$ and $({\go g}_{0}, V^*)$. Similarly, for $U\in {\go g}_{0}, X\in V, Y\in V^*$ we will often write $U.X$ and $U.Y$ instead of $\rho(U)X$ and $\rho^*(U)Y$.
 Put ${\go g}_{-1}=V^*$ and ${\go g}_{1}=V$. Define also
 $$\Gamma({\go g}_{0},B_0,\rho)= {\go g}_{-1}\oplus {\go g}_{0}\oplus {\go g}_{1}=V^*\oplus {\go g}_{0}\oplus V.$$
 
 Our aim is now to define a structure of local Lie algebra on $\Gamma({\go g}_{0},B_0,\rho)$, such that for $U\in {\go g}_{0}, X\in {\go g}_{1}, Y\in {\go g}_{-1}$, we have $[U,X]=U.X$ and $[U, Y]=U.Y$.

  \begin{definition} \label{def-fundamental-triplets} \hfil

 We call the triplet $(\go{g}_{0}, B_{0}, (\rho,V))$ a fundamental triplet. It consists of the following ingredients: 
 
 {\rm a)} a quadratic Lie algebra $(\go{g}_{0}, B_{0})$.

{ \rm b)} a finite dimensional representation $(\rho,V)$ of $\go{g}_{0}$ on the space $V$.

In order to simplify the notation, we will will sometimes denote the triplet by  $(\go{g}_{0}, B_{0}, \rho)$ or $(\go{g}_{0}, B_{0}, V)$.
  \end{definition}

 \begin{theorem}\label{th-algebre-locale}\hfill
 
 Let $({\go g}_{0},B_{0},\rho)$ be a fundamental triplet.  As before we set:
 $$\Gamma({\go g}_{0},B_0,\rho)= {\go g}_{-1}\oplus {\go g}_{0}\oplus {\go g}_{1}=V^*\oplus {\go g}_{0}\oplus V.$$
 For $U\in {\go g}_{0}, X\in {\go g}_{1}, Y\in {\go g}_{-1}$  define an anticommutative bracket by

 $ a)$ $[U,X]=U.X$, $[U,Y]=U.Y$
 
 $b)$ The element $[X,Y]$ is the unique element of ${\go g}_{0}$ such that for all $U\in {\go g}_{0}$ the following identity holds:
 $$B_{0}([X,Y],U)=Y(U.X)=-(U.Y)(X).$$
  $($The last equality is just the definition of the contragredient representation$)$.
 
 The preceding  bracket defines a structure of a local Lie algebra on $\Gamma({\go g}_{0},B_0,\rho)$.

 \end{theorem}
 \begin{proof} We must prove that the Jacobi identity is verified each time the brackets make sense. This means that we have to prove the following identities.
 
   $\alpha)$ $\forall \, U_{1}, U_{2}\in {\go g}_{0}$, $\forall X\in {\go g}_{1}$,
 $[X,[U_{1},U_{2}]]=[[X,U_{1}],U_{2}]+[U_{1},[X,U_{2}]]$.
 
   $\beta)$ $\forall \, U_{1}, U_{2}\in {\go g}_{0}$, $\forall Y\in {\go g}_{-1}$,
 $[Y,[U_{1},U_{2}]]=[[Y,U_{1}],U_{2}]+[U_{1},[Y,U_{2}]]$.
 
    $\gamma)$ $\forall \, X\in {\go g}_{1}$, $\forall \, Y\in {\go g}_{-1}$, $\forall \, Z\in {\go g}_{0}$, $[Z,[X,Y]]=[[Z, X],Y]+[X,[Z, Y]]$.
   \vskip 3pt
   We have:
  
   $[X,[U_{1},U_{2}]]=-[[U_{1},U_{2}],X]=-[U_{1},U_{2}].X=-U_{1}.(U_{2}.X)+U_{2}.(U_{1}.X)$.
   
   On the other hand we have:
   
   $[[X,U_{1}],U_{2}]= [U_{2},[U_{1},X]]=U_{2}.(U_{1}.X)$ and $[U_{1},[X,U_{2}]]=-[U_{1},[U_{2},X]]=-U_{1}.(U_{2}.X)$.
   This proves $\alpha).$ The proof of the identity $\beta)$ is similar.
   
 Let us now consider the identity $\gamma)$. We set $L=[Z,[X,Y]]$, $R_{1}=[[Z, X],Y]$, $R_{2}=[X,[Z, Y]]$, and $R=R_{1}+R_{2}$. As $L,R_{1},R_{2},R\in {\go g}_{0}$, in order to prove $\gamma)$ it will be enough to show that for all $U\in {\go g}_{0}$, we have $B_{0}(L,U)=B_{0}(R,U)$. Using the invariance of $B_{0}$ and  definition $b)$ we get:
 
 $\begin{array}{rl}B_{0}(L,U)&=B_{0}([Z,[X,Y]],U)=-B_{0}([[X,Y],Z],U)= -B_{0}([X,Y],[Z,U])\\
 {}&= -Y([Z,U].X)\\
 {}&=-Y(Z.(U.X)-U.(Z.X)).
 \end{array}$

 On the other hand, using again definition $b)$, we have also:

$ \begin{array}{rl}
B_{0}(R_{1},U)&=B_{0}([[Z, X],Y],U)= Y(U.[Z,X])=Y(U.(Z.X)) 
\end{array}$

and $ \begin{array}{rl}
B_{0}(R_{2},U)&=B_{0}([X,[Z, Y]],U)= [Z,Y](U.X)=Z.Y(U.X)=-Y(Z.(U.X)).
\end{array}$

Hence $B_{0}(R,U)=Y(U.(Z.X)-Z.(U.X))=B_{0}(L,U).$

 \end{proof}

 Hence we have associated a local Lie algebra to the data $\go{g}_{0}, B_{0}, (\rho, V)$.

 \begin{notation}\label{not-Gamma(V)}
 For convenience we will sometimes denote this local algebra by $\Gamma(\go{g}_{0},B_{0},V) $ instead of $\Gamma(\go{g}_{0},B_{0},\rho)$.
 \end{notation}

\vskip 5pt
\subsection {Isomorphisms of the local parts and dependence on $(B_{0},\rho)$} \hfill
 \vskip5pt
    We will first  determine  necessary and sufficient conditions on the data $(\go{g}_{0}^i, B_{0}^i,\rho_{i})$ $(i=1,2)$ for the corresponding  local Lie algebras to be isomorphic. 
   
    Let $\Psi:\Gamma(\go{g}_{0}^1, B_{0}^1,\rho_{1})\longrightarrow \Gamma(\go{g}_{0}^2, B_{0}^2,\rho_{2})$ be an graded isomorphism between the two underlying vector spaces. Set $A= \Psi_{|_{\go{g}_{0}^1}}: \go{g}_{0}^1\longrightarrow \go{g}_{0}^2$,  $\gamma=\Psi_{|_{V_{1}}}:V_{1}\longrightarrow V_{2}$ and $\widetilde\gamma=\Psi_{|_{V_{1}^*}}:V_{1}^*\longrightarrow V_{2}^*$. With these notations we will set $\Psi= (\widetilde\gamma,A, \gamma)$. Our aim is  to determine under which conditions on   $\widetilde\gamma,A, \gamma$, the map $\Psi$ is an isomorphism of local Lie algebras. An obvious condition is of course that $A$ is an isomorphism of Lie algebras.
    
    \begin{prop}\label{prop-CNS-iso}\hfill

    Let $\widetilde\gamma :V_{1}^*\longrightarrow V_{2}^*$ and $ \gamma :V_{1}\longrightarrow V_{2}$ be two isomorphisms of vector spaces and let  $A: \go{g}_{0}^1\longrightarrow \go{g}_{0}^2$ be an isomorphism of Lie algebras. Then $$\Psi=(\widetilde\gamma,A, \gamma): \Gamma(\go{g}_{0}^1, B_{0}^1,\rho_{1})\longrightarrow \Gamma(\go{g}_{0}^2, B_{0}^2,\rho_{2})$$
    is an isomorphism of local Lie algebras if and only the following three conditions hold $($for all $ U_{1}\in \go{g}_{0}^1,X_{1}\in V_{1},  Y_{1}\in V_{1}^*$$ ):$
    
    \hskip 20pt $1)$\hskip 5pt $ \rho_{2}(A(U_{1}))\circ\gamma=\gamma\circ\rho_{1}(U_{1})$,
    
    \hskip 20pt $2)$\hskip 5pt $ \rho_{2}^*(A(U_{1}))\circ \widetilde \gamma=\widetilde \gamma\circ\rho_{1}^*(U_{1})$,

      \hskip 20pt $3)$\hskip 5pt $B_{0}^2(A([X_{1},Y_{1}]), A(U_{1}))=B_{0}^1([^t\widetilde \gamma\circ\gamma(X_{1}),Y_{1}], U_{1})$.
    \end{prop}
    
    \begin{proof}\hfill
    
   $ \Psi=(\widetilde\gamma,A, \gamma)$ is an isomorphism if and only if the following three conditions hold for all $ U_{1}\in \go{g}_{0}^1,X_{1}\in V_{1},  Y_{1}\in V_{1}^*$: 
   
   \hskip 20pt  $1')$\hskip 5pt $\gamma([U_{1},X_{1}])=[A(U_{1}),\gamma(X_{1})]$,
  
  \hskip 20pt  $2')$\hskip 5pt $\widetilde\gamma([U_{1},Y_{1}])=[A(U_{1}),\widetilde\gamma(Y_{1})]$,
  
   \hskip 20pt  $3')$\hskip 5pt $A([X_{1},Y_{1}])=[\gamma(X_{1}),\widetilde\gamma(Y_{1})]$.
     
     But from the definition of the brackets in $\Gamma(\go{g}_{0}^1, B_{0}^1,\rho_{1})$ and $\Gamma(\go{ g}_{0}^2, B_{0}^1,\rho_{2})$, conditions $1')$ and $2')$    are equivalent to conditions $1)$ and $2)$ respectively.
     
     Again from the definition of the brackets $3')$  is equivalent to 
      $$B_{0}^2(A[X_{1},Y_{1}],U_{2})=B_{0}^2([\gamma(X_{1}), \widetilde \gamma(Y_{1})],U_{2})$$
     for all $X_{1}\in V_{1},Y_{1}\in V_{1}^*, U_{2}\in \go{g}_{0}^2$. Setting  $U_{2}=A(U_{1})$ and using the transpose of condition $2)$    leads to condition $3)$.

    \end{proof}

  Now we give a specific criterion in the case where the representation $\rho_{1}$ is  irreducible.
  
  \begin{prop}\label{prop-CNS-iso-irreductible}\hfill

   Suppose that the representation $(\go{g}_{0}^1,\rho_{1},V_{1})$ is  irreducible.  
   
   Let $\widetilde\gamma :V_{1}^*\longrightarrow V_{2}^*$ and $ \gamma :V_{1}\longrightarrow V_{2}$ be two isomorphisms of vector spaces and let  $A: \go{g}_{0}^1\longrightarrow \go{g}_{0}^2$ be an isomorphism of Lie algebras. Then $$\Psi=(\widetilde\gamma,A, \gamma): \Gamma(\go{g}_{0}^1, B_{0}^1,\rho_{1})\longrightarrow \Gamma(\go{g}_{0}^2, B_{0}^2,\rho_{2})$$
    is an isomorphism of local Lie algebras if and only if the following two conditions hold $($for all $ U_{1}\in \go{g}_{0}^1,X_{1}\in V_{1},  Y_{1}\in V_{1}^*$$ ):$
    
        \hskip 20pt $1)$\hskip 5pt $ \rho_{2}(A(U_{1}))\circ\gamma=\gamma\circ\rho_{1}(U_{1})$,
    
    \hskip 20pt $2)$\hskip 5pt There exists $c\in \C^*$ such that:
    
    \hskip 30pt  a) \hskip 5pt  $^t\widetilde\gamma\circ\gamma= c \text{Id}_{V_{1}}$
    
    \hskip 30pt  b)  \hskip 5pt $B_{0}^2(A([X_{1},Y_{1}]), A(U_{1}))= c B_{0}^1([ X_{1},Y_{1}], U_{1})$.

  \end{prop}
  
  \begin{proof}
  Condition $1)$ appears already  in Proposition \ref{prop-CNS-iso}.
   Composing by $\gamma$ on the right condition $2)$ in this previous Proposition  leads to 
  $$^t\widetilde\gamma\circ \gamma\circ \rho_{1}(U_{1})=\rho_{1}(U_{1})\circ ^t\widetilde\gamma\circ \gamma.$$
    By Schur's lemma we get $^t\widetilde\gamma\circ\gamma= c \text{Id}_{V_{1}}$, with $c\in \C^*$.  Then condition $2) b)$ above is just a rewriting of condition $3) $ of Proposition \ref{prop-CNS-iso}.
  
  Conversely assuming  conditions that $1)$  $2)a)$, and  $2) b)$ above hold it is easy to recover conditions $1)$, $2)$ and $3)$ in the previous Proposition (for example from $2) a)$, condition $2)$ of \ref{prop-CNS-iso} is just the transpose of $1)$).    
  \end{proof}

  \begin{definition}  \label{def-morphism-triplets}\hfill
  
  Let $(\go{g}_{0}^1, B_{0}^1, \rho_{1})$ and $(\go{g}_{0}^2, B_{0}^2, \rho_{2})$ be two fundamental triplets. Let $A\in \Hom(\go{g}_{0}^1,\go{g}_{0}^2)$ is a isomorphism of Lie algebras and let $\gamma\in \Hom(V_{1},V_{2})$ be an isomorphism of vector spaces. We say that the pair $(A,\gamma)$ is an isomorphism of fundamental triplets if 
  
  ${\rm a)}$  
  $$ \forall U,\, \forall V \in \go{g}_{0}^1,\hskip 5pt B_{0}^2(A(U),A(V))=B_{0}^1(U,V)  $$
  
  ${\rm b)}$ 
  $$\forall U\in \go{g}_{0}^1, \hskip 5pt\rho_{2}(A(U))\circ \gamma=\gamma\circ \rho_{1}(U) $$
  

  \end{definition}
  
  \begin{rem} In the definition above, condition ${\rm a)} $ coincides with the notion of isometric isomorphism, or $i$-isomorphism of quadratic Lie algebras introduced in \cite{Duong-et-al}.
  This notion of $i$-isomorphism was already implicit in \cite{Medina-Revoy}.  \end{rem}
  
  \begin{theorem}\label{th-extension-iso}\hfill
  
  
  Any isomorphism  of fundamental triplets 
  $$(A,\gamma): (\go{g}_{0}^1, B_{0}^1, \rho_{1}) \longrightarrow (\go{g}_{0}^2, B_{0}^2, \rho_{2})$$
    extends  to an isomorphism of local Lie algebras 
    $$\Psi_{(A,\gamma)}:  \Gamma({\go g}_{0}^1,B_0^1,\rho_{1})\longrightarrow \Gamma({\go g}_{0}^2,B_0^2,\rho_{2}).$$
    
    Moreover, if   the space $\langle \rho_{1}(\go{g}_{0}^1). V_{1}\rangle $ generated by the vectors $\rho_{1}(U)X$ $(X\in V_{1}, U\in \go{g}_{0}^1)$ is equal to $V_{1}$, then the preceding extension is unique.
    

  \end{theorem}
  
  \begin{proof}\hfill
  
  Note first that condition $b)$  in Definition \ref{def-morphism-triplets} is exactly condition $1)$ in Proposition \ref{prop-CNS-iso}. Define $\widetilde\gamma: V_{1}^*\longrightarrow V_{2}^*$ by $\widetilde\gamma=^t\kern -5pt\gamma^{-1}$.   Then it is easy to check that 
  $$\forall U\in \go{g}_{0}^1,\hskip 5pt \rho_{2}^*(A(U))\circ \widetilde\gamma=\widetilde\gamma\circ \rho_{1}^*(U)  $$
  and this is exactly condition $2)$ in Proposition \ref{prop-CNS-iso}.
  
  As $^t \widetilde\gamma\circ\gamma=\text{Id}_{|_{V_{1}}}$, condition $a)$ in Definition  \ref{def-morphism-triplets} is exactly condition 3) in Proposition  \ref{prop-CNS-iso}. This Proposition implies then that $\Psi_{(A,\gamma)}= ( \widetilde\gamma,A,\gamma)$ is an isomorphism of the corresponding local Lie algebras.

     It remains to prove the uniqueness of the extension under the condition $\langle \rho_{1}(\go{g}_{0}^1). V_{1}\rangle =V_{1}$ . But if 
   $$\Psi:  \Gamma({\go g}_{0}^1,B_0^1,V_{1})\longrightarrow \Gamma({\go g}_{0}^2,B_0^2,V_{2})$$
   is an isomorphism of local Lie algebras such that $\Psi_{|_{\go{g}_{0}}}=A$,  $\Psi_{|_{V_{1}}}=\gamma$ and  $\Psi_{|_{V_{1}^*}}=\widetilde{\gamma}:V_{1}^*\longrightarrow V_{2}^*$, then for $X_{1}\in V_{1}, Y_{1}\in V_{1}^*$, $A([X_{1},Y_{1}])=[\gamma(X_{1}),\widetilde{\gamma}(Y_{1})]$. Therefore 
   $$B_{0}^2(A([X_{1},Y_{1}]),U_{2})=B_{0}^2([\gamma(X_{1}),\widetilde{\gamma}(Y_{1})],U_{2}), $$ for all $U_{2}\in \go{g}_{0}^2, \,X_{1}\in V_{1}, Y_{1}\in V_{1}^*$.  Then for $U_{2}=A(U_{1})$, $U_{1}\in V_{1}$, 
   
   $\begin{array}{lll}
B_{0}^2(A([X_{1},Y_{1}]),A(U_{1}))&= B_{0}^1([X_{1},Y_{1}], U_{1})&=Y_{1}(\rho_{1}(U_{1})X_{1})\\

=B_{0}^2([\gamma(X_{1}),\widetilde{\gamma}(Y_{1})],A(U_{1}))&= \widetilde{\gamma}(Y_{1})(\rho_{2}(A(U_{1}))\circ \gamma(X_{1}))&= Y(^t\widetilde{\gamma} \circ \rho_{2}(A(U_{1}))\circ\gamma(X_{1}))\\
= Y_{1}(^t\widetilde{\gamma} \circ \gamma\circ\rho_{1}(U_{1})X_{1})&&
   
   \end{array}$
   
 Then from $Y_{1}(\rho_{1}(U_{1})X_{1})=  Y_{1}(^t\widetilde{\gamma} \circ \gamma\circ\rho_{1}(U_{1})X_{1})$, we obtain that $^t\widetilde{\gamma} \circ \gamma $ is the identity on the space $\langle \rho_{1}(\go{g}_{0}^1). V_{1}\rangle=V_{1}$. Therefore $\widetilde{\gamma} = \,^t\kern-1pt\gamma^{-1}$.

  \end{proof}

   The local Lie algebra $\Gamma({\go g}_{0},B_0,\rho)$  may depend strongly  on the fundamental triplet, see Example  \ref{ex-sl2n} below.   
   
 To illustrate this, we will now   investigate the modification  of the bracket in  $\Gamma(\go g_{0}, B_{0}, \rho)$ under a slight change of $\rho$.   
 
 
 


Suppose  that ${\go g}_{0}={ Z}\oplus L$ is a quadratic Lie algebra where $Z$ is a central ideal and $L$ is an ideal.  For $\gamma \in \C^*$ we denote by  $\gamma$${\scriptscriptstyle\Box}$$\rho$  the representation of ${\go g}_{0}$ on $V$ given by $\gamma{\scriptscriptstyle\Box}\rho(z+u)=\gamma\rho(z)+\rho(u)$, for $z\in Z$ and $u\in L$. If  $U=z+u\in {\go g}_{0}$,  we also set $\gamma{\scriptscriptstyle\Box}U=\gamma z+u$. Then $\gamma{\scriptscriptstyle\Box}\rho(U)=\rho( \gamma{\scriptscriptstyle\Box}U)$. If $B_{0}=B_{0, Z}+B_{0,L}$ where $B_{0,Z}$ and $B_{0, L}$ are forms on $Z$ and $L$ respectively, we  define $\gamma {\scriptscriptstyle\Box} B_{0}= \gamma B_{0,Z}+B_{0,L}$.


The next proposition indicates the dependance of the local Lie algebra $\Gamma({\go g}_{0},B_0,\rho)$ if we change $\rho$ into $\gamma{\scriptscriptstyle\Box}\rho$.

\begin{prop}\label{prop-changement-rep}\hfill

{\rm 1)} Let us  denote by $[\,\, , \,\,]_{B_{0},\rho}$ the bracket on $\Gamma({\go g}_{0},B_0,V)$ given by  Theorem \ref{th-algebre-locale}. Then, using the notations defined above, we have:
$$[X,Y]_{B_{0}, \gamma{\scriptscriptstyle\Box}\rho}=\gamma{\scriptscriptstyle\Box}[X,Y]_{B_{0},\rho}$$
$$\,[U,X]_{B_{0},\gamma {\scriptscriptstyle\Box}\rho}=\gamma {\scriptscriptstyle\Box} \rho(U)X,\,\,  [U,Y]_{B_{0},\gamma {\scriptscriptstyle\Box}\rho}=-\gamma {\scriptscriptstyle\Box} \rho^*(U)Y  $$
{\rm 2)} Suppose that $\lambda,\mu,\alpha,\beta \in \C^*$ verify the condition $\frac{\mu^2}{\lambda}=\frac{\beta^2}{\alpha}$. Then the local Lie algebras $\Gamma({\go g}_{0}, \lambda {\scriptscriptstyle\Box}B_{0}, \mu {\scriptscriptstyle\Box} \rho)$ and $\Gamma({\go g}_{0},\alpha {\scriptscriptstyle\Box}B_{0}, \beta {\scriptscriptstyle\Box} \rho)$ are isomorphic. In particular $\Gamma({\go g}_{0}, \lambda {\scriptscriptstyle\Box}B_{0},  \rho)$ and $\Gamma({\go g}_{0},  B_{0}, \frac{1}{\sqrt{\lambda}} {\scriptscriptstyle\Box} \rho)$ are isomorphic.
\end{prop}

\begin{proof}\hfill

1) The element $[X,Y]_{B_{0}, \gamma{\scriptscriptstyle\Box}\rho}$ is by definition the unique element of ${\go g}_{0}$ such that, for all $U\in {\go g}_{0}, $ $B_{0}([X,Y]_{B_{0}, \gamma{\scriptscriptstyle\Box}\rho},U)= Y( \gamma{\scriptscriptstyle\Box}\rho(U)X)= Y(\rho(\gamma {\scriptscriptstyle\Box }U)X)
=B_{0}([X,Y]_{B_{0},\rho}, \gamma{\scriptscriptstyle \Box }U) =B_{0}(\gamma{\scriptscriptstyle \Box}[X,Y]_{B_{0},\rho},U) $. Hence $[X,Y]_{B_{0}, \gamma{\scriptscriptstyle\Box}\rho}=\gamma{\scriptscriptstyle \Box}[X,Y]_{B_{0},\rho}$. The other two identities are just the definitions of $[U,X]_{B,\gamma {\scriptscriptstyle\Box}\rho}$  and $[U,Y]_{B,\gamma {\scriptscriptstyle\Box}\rho}$.

2) The criterion of Proposition \ref{prop-CNS-iso} shows that the map $\varphi: \Gamma({\go g}_{0}, \lambda {\scriptscriptstyle\Box}B_{0}, \mu {\scriptscriptstyle\Box} \rho)\longrightarrow \Gamma({\go g}_{0},\alpha {\scriptscriptstyle\Box}B_{0}, \beta {\scriptscriptstyle\Box} \rho)$  defined by $\varphi(U)=\frac{\mu}{\beta}{\scriptscriptstyle\Box}U \, (U\in {\go g}_{0}),    \varphi(X)=X\, (X\in V),\varphi(Y)=Y   \, (Y\in V^*)$ is an isomorphism of local Lie algebras.

\end{proof}

  \begin{rem}\label{rem-retournement}  {\bf (inverse isomorphism)}\hfill
  
 Let $(\go{g}_{0},B_{0})$ be a quadratic Lie algebra. Consider a finite dimensional representation $(\rho,V)$ of  $\go{g}_{0}$. Set $\go{g}_{0}^1= \go{g}_{0}^2= \go{g}_{0}$,  $\go{g}_{1}^1=V$, $\go{g}_{-1}^1=V^*$, $\go{g}_{1}^2=V^*$, $\go{g}_{-1}^2=V$. Hence $\Gamma(\go{g}_{0}^1, B_{0}, \rho)= V^*\oplus \go{g}_{0}^1\oplus V$ and $\Gamma(\go{g}_{0}^2, B_{0}, \rho^*)= V\oplus \go{g}_{0}^2\oplus V^*$.
  The map 
  $$\theta: \Gamma(\go{g}_{0}, B_{0}, \rho)\longrightarrow \Gamma(\go{g}_{0}, B_{0}, \rho^*)$$
  defined by $\theta(U)=U$ for all $U\in \go{g}_{0}^1= \go{g}_{0}^2= \go{g}_{0}$, $\theta(X)=X$ for all $X\in \go{g}_{1}^1=V$ and $\theta(Y)=Y$ for all $Y\in \go{g}_{0}^2= V^*$, is a {\it non graded } isomorphism  of local Lie algebras (it sends $\go{g}_{1}^1$ onto $\go{g}_{-1}^2$ and  $\go{g}_{-1}^1$ onto $\go{g}_{1}^2$).
  Of course the isomorphism $\theta$ extends uniquely to an  non graded isomorphism, still denoted $\theta$,  between $\go{g}_{max}(\Gamma(\go{g}_{0}, B_{0}, \rho))$ and $\go{g}_{max}(\Gamma(\go{g}_{0}, B_{0}, \rho^*))$.

  \end{rem}

\vskip 5pt
\subsection { Graded Lie algebras with local part $\Gamma({\go g}_{0},B_0,\rho)$} \hfill
 \vskip5pt
 
  Let $(\go{g}_{0},B_{0},\rho)$ be a fundamental triplet. Let $\Gamma(\go{g}_{0},B_{0},\rho)$ be the local Lie algebra constructed in Theorem  \ref{th-algebre-locale}. From Theorem \ref{th-Kac} we can associate two graded Lie algebras  $\go{g}_{max}( \Gamma({\go g}_{0},B_0,\rho))$ and $\go{g}_{min}( \Gamma({\go g}_{0},B_0,\rho))$ to these data.

  \vskip 5pt

  
   \begin{rem} \label{rem-decomposition}
   Let $\Gamma$ and $\Gamma'$ be two local Lie algebras. Then $\Gamma\times \Gamma'$ is naturally a local Lie algebra, and is is easy to see from Theorem \ref{th-Kac} that $\go{g}_{min}(\Gamma\times \Gamma')\simeq \go{g}_{min}(\Gamma)\times \go{g}_{min}(\Gamma')$. The same statement for the maximal  algebras is not true in general. However if $\Gamma'= \go{g}_{0}'$ is concentrated in degree $0$, one has $\go{g}_{max}(\Gamma\times \Gamma')\simeq \go{g}_{max}(\Gamma)\times \Gamma'$.
      
      Let $(\go{g}_{0},B_{0})$ be a quadratic Lie algebra. Suppose that $\go{g}_{0}=\go{g}_{0}^1\oplus \go{g}_{0}^2$ is an orthogonal decomposition into ideals. Define  $B_{0}^1=B_{{0}_{|_{\go{g}_{0}^1\times \go{g}_{0}^1}}}$ and $B_{0}^2=B_{{0}_{|_{\go{g}_{0}^2\times \go{g}_{0}^2}}}$. Suppose also that the representation $(\go{g},\rho,V)$ is a direct sum $(\go{g}_{0}^1\oplus \go{g}_{0}^2, \rho_{1}\oplus \rho_{2}, V_{1}\oplus V_{2}$). From the definitions we obtain that $\Gamma(\go{g}_{0},B_{0},V)=\Gamma(\go{g}_{0}^1,B_{0}^1,V_{1})\times  \Gamma(\go{g}_{0}^2,B_{0}^2,V_{2})$, and then from the previous discussion we get:
   $$\go{g}_{min}( \Gamma({\go g}_{0},B_0,\rho))\simeq \go{g}_{min}( \Gamma({\go g}_{0}^1,B_0^1,\rho_{1}))\times \go{g}_{min}( \Gamma({\go g}_{0}^2,B_0^2,\rho_{2})).$$
 \end{rem}
\vskip20pt
 
 As an example let us show that graded reductive (finite dimensional) Lie algebras are always minimal graded Lie algebras. 
 
 \begin{prop}\label{prop-reductive=minimal}\hfill
 
 Let ${\go g}$ be a graded  Lie algebra.  Let $B_{\Gamma(\go{g})}$ be a nondegenerate invariant symmetric bilinear form on $\Gamma(\go{g})$ such that $B_{\Gamma(\go{g})}(\go{g}_{i},\go{g}_{j})=0$ when $|i+j|\neq 0$. Then, using $B_{\Gamma(\go{g})}$, the contragredient representation $(\go{g}_{0}, \go{g}_{1}^*)$ can be identified with $(\go{g}_{0}, \go{g}_{-1})$.  If $B_{0}$ denotes the restriction of $B_{\Gamma(\go{g})}$ to $\go{g}_{0}$, then we have:
 $$ \go{g}_{min}(\Gamma(\go{g}))\simeq \go{g}_{min}(\Gamma(\go{g}_{0},B_{0}, \go{g}_{1})).$$
 
 Moreover suppose that ${\go g}$ is a reductive $($finite dimensional$)$ Lie algebra. Suppose that we are given a  $\Z$-grading  ${\go g}=\oplus_{i\in \Z} {\go g}_{i}$ such that ${\go g}$ is generated by its local part $\Gamma({\go g})={\go g}_{-1}\oplus {\go g}_{0}\oplus {\go g}_{1}$. Then 
  $$\go{g}\simeq  \go{g}_{min}(\Gamma(\go{g}))\simeq \go{g}_{min}(\Gamma(\go{g}_{0},B_{0}, \go{g}_{1})).$$
 \end{prop}
 
 \begin{proof} \hfill
 
 From the definition we can identify $\go{g}_{1}^*$ with $\go{g}_{-1}$ by using $B_{\Gamma(\go{g})}$. Then for $U\in \go{g}_{0}$, $Y\in \go{g}_{-1}$ and $X\in \go{g}_{1}$ we have, from the invariance of $B_{\Gamma(\go{g})}$:
 $$B_{\Gamma(\go{g})}(U.Y,X)=-B_{\Gamma(\go{g})}(Y,[U,X])=B_{\Gamma(\go{g})}([U,Y],X)$$
 and hence $[U,Y]=U.Y$ (here $U.Y$ stands for the contragredient action of $U\in \go{g}_{0}$ on $\go{g}_{-1}\simeq \go{g}_{1}^*$).
 
 Similarly, we have also:
 $$B_{0}([X,Y],U)=B_{\Gamma(\go{g})}([X,Y],U)=B_{\Gamma(\go{g})}(Y,[U,X])=Y(U.X).$$ 
 
Hence the original bracket in $\Gamma(\go{g})$ is the bracket constructed in Theorem \ref{th-algebre-locale}. Therefore $\go{g}_{min}(\Gamma(\go{g}))\simeq \go{g}_{min}(\Gamma(\go{g}_{0},B_{0}, \go{g}_{1}))$.

 Suppose now that $\go{g}$ is reductive graded. From the universal property of $\go{g}_{min}(\Gamma(\go{g}))$ there exists a graded ideal $I\subset  {\go g}$ such that ${\go g}/I\simeq \go{g}_{min}(\Gamma(\go{g}))$. As ${\go g}$ is reductive there exists an ideal $U\subset {\go g}$ such that ${\go g}=U\oplus I$. Hence $U\simeq\go{g}_{min}(\Gamma(\go{g}))$. But then the local part of $U$ is $\Gamma(\go{g})$, and this contradicts the fact that ${\go g}$ is generated by $\Gamma(\go{g})$ unless $I=\{0\}$.  
 
 \end{proof}
 
 \begin{example}  \label{ex-PH-paraboliques}{\bf (Prehomogeneous vector spaces of parabolic type)}\hfill
 
   It must be noticed that if  $\go{g}=\oplus _{i=-n}^{n}\go{g}_{i}$ is an arbitrary  $\Z$-grading of a semi-simple Lie algebra $\go{g}$  then $\go{g}$ is in general not  generated by its local part, and is therefore not a graded Lie algebra in the sense of section 2.1. Let us explain this briefly. It is well known that there exists always a  grading element, that is an element $H\in \go{g}$ such that $\go{g}_{i}=\{X\in \go{g}\,|\, [H,X]=iX\}$. Let $\go{h}$ be Cartan subalgebra of $\go{g}_{0}$ (which is also a Cartan subalgebra of $\go{g}$), containing $H$. Let $\Psi$ be a set of simple roots of the root system $\Sigma(\go{g},\go{h})$ such that $\alpha(H)\in \N$ (such a set of simple roots always exists). Hence we have associated  a "weighted Dynkin diagram" to a grading  . The subdiagram of roots of weight 0 corresponds to the semi-simple part of the Levi subalgebra $\go{g}_{0}$. But if the weighted Dynkin diagram has weights equal to $1$ and to  $n>1$, then it is easy to see that $\go{g}$ cannot be generated by the local part $\go{g}_{-1}\oplus \go{g}_{0}\oplus \go{g}_{1}$.

We explain  now how to construct gradings $\go{g}=\oplus _{i=-n}^{n}\go{g}_{i}$ such that $\go{g}$ is generated by the local part. Let $\go{h}$ be an arbitrary  Cartan subalgebra of $\go{g}$. As before  $\Sigma(\go{g},\go{h})$ is the set of roots of the pair $(\go{g},\go{h})$ and  $\Psi$ is a set of simple roots. Let $\theta\subset \Psi$ be a subset and let $<\theta>$ denote the subset of roots which are linear combinations of elements of $\theta$. Let $\go{g}_{0}=\go{l}_{\theta}$ be the Levi subalgebra corresponding to $\theta$. That is $\go{g}_{0}=\go{h}\oplus (\oplus_{\alpha\in<\theta>}\go{g}^{\alpha})$ where $\go{g}^{\alpha}$ is the root space corresponding to $\alpha$. Let $H$ be the unique element in $\go{h}$ such that $\alpha(H)=0$ if $\alpha\in \theta$ and $\alpha(H)=1$ if $\alpha \in \Psi\setminus \theta$. Define then $\go{g}_{i}=  \{X\in \go{g}\,|\, [H, X]=iX\} $ (this definition of $\go{g}_{0}$ is coherent with the preceding one). One obtains this way a grading $\go{g}=\oplus _{i=-n}^n\go{g}_{i}$. The representations $(\go{g}_{0}, \go{g}_{1})$ are prehomogeneous vector spaces called prehomogeneous spaces of parabolic type. It is easy to see that they correspond to gradings whose weights in the sense above are only $0$ and $1$. From Proposition \ref{prop-reductive=minimal} we obtain that in this case $\go{g}=\go{g}_{min}(\Gamma(\go{g}_{0},B_{0}, \go{g}_{1}))$ where $B_{0}$ is the restriction of the Kiling form of $\go{g}$ to $\go{g}_{0}$.  
 
 \end{example}

  \begin{example}\label{ex-Kac-Moody}{\bf (Principal gradings of symmetrizable  Lie algebras)}\hfill
  
  We adopt here the same definitions and notations as in \cite{Kac3}. Let $A=(a_{i,j})$ be a $n\times n$ matrix with complex coefficients and of rank $\ell$. Let $(\go{h}, \Pi,\check{\Pi})$ be a realization of $A$.  This means that $\go{h}$ is a complex vector space,  $\Pi=\{\alpha_{1},\dots,\alpha_{n}\}\subset \go{h}^*$ and $\check{\Pi}=\{{\check{\alpha}}_{1},\dots,{ \check{\alpha}}_{n}\}\subset \go{h}$ are subsets of $\go{h}^*$ and $\go{h}$ respectively, satisfying the following conditions:  $\Pi$ and $\check{\Pi}$ are linearly independent, $\langle \check{\alpha}_{i}, \alpha_{j}\rangle= a_{i,j}$ and $n-\ell=\dim \go{h} -n$.  The Lie algebra $\widetilde{\go{g}}(A)$ is the Lie algebra with generators $\go{h}$, $e_{i}$, $f_{i}$ ($i=1,\dots,n$) and  the relations:
  $$[e_{i},f_{j}]=\delta_{i,j} \check{\alpha}_{i},  \,\, [h,e_{i}]=\alpha_{i}(h)e_{i}, [h,f_{i}]=-\alpha_{i}(h)f_{i},\,\, \,[h,h']=0, \,\,(h,h'\in \go{h}).$$
  Then the Lie algebra  $\go{g}(A)$ is defined by $ \widetilde{\go{g}}(A)/\go{r}$ where $\go{r}$ is the maximal ideal intersecting $\go{h}$ trivially. And then $\go{g}(A)= \bigoplus_{i\in \Z}\go{g}_{i}$ where $\go{g}_{i}=\bigoplus_{ht(\alpha)=i}\go{g}_{\alpha}$, and where $\go{g}_{\alpha}$ is the root space of $\alpha$ (if $\alpha=\sum_{i=1}^n m_{i}\alpha_{i}, \text{ht}(\alpha)=\sum_{i}m_{i}$). This is the principal grading of $\go{g}(A)$.  We have $\go{g}_{-1}=\oplus_{i=1}^n\C f_{i} $, $\go{g}_{0}=\go{h}$, $\go{g}_{1}=\oplus_{i=1}^n\C e_{i} $, and it is easy to see that $\go{g}(A)=\go{g}_{min}(\Gamma(A))$ where $\Gamma(A)= \go{g}_{-1}\oplus \go{g}_{0}\oplus \go{g}_{1}$.
  
    We suppose further that the matrix $A$ is symmetrizable. This means that there exist a symmetric matrix $B=(b_{i,j}) $ and an invertible  diagonal matrix $D=\text{diag}(\epsilon _{1},\dots,\epsilon_{n})$ such that $A=DB$. The Lie algebra is then called a {\it symmetrizable Lie Algebra}. Then, according to Theorem 2.2. of \cite{Kac3}, there exists a non-degenerate invariant symmetric bilinear form $(.,.)$ on $\go{g}(A)$ such that 
    the restriction of this form to $\go{h}$ is non-degenerate and 
    $(\go{g}_{i},\go{g}_{j})=0$ if $i+j\neq0$. Therefore $\go{g}_{-i}$ can be identified to $\go{g}_{i}^*$ and the representation $(\go{g}_{0},\go{g}_{-1}) $ is the contragredient representation of $(\go{g}_{0},\go{g}_{1}) $. 
    Let $\rho$ denote the representation  $(\go{g}_{0},\go{g}_{1})$ and let $B_{0}$ be the restriction of the form $(.,.)$ to $\go{g}_{0}=\go{h}$. Then, in the notations of section $3.3$ we have 
    $$\go{g}(A)= \go{g}_{min}(\Gamma(\go{g}_{0}, B_{0},\rho)).$$


  
 

   \end{example}

 \begin{example}\label{ex-sl2n}

 We will now examine the case of ${\go {sl}}_{2n}(\C)$ which will be considered both as a graded Lie algebra and a local Lie algebra. This will show that   the local Lie algebras  $\Gamma({\go g}_{0},B_0,\rho)$ and the corresponding minimal Lie algebra $\go{g}_{min}(\Gamma( {\go g}_{0},B_{0},\rho))$ depend strongly on the choice of $B_{0}$.
 
 Consider first the classical grading of ${\go {sl}}_{2n}(\C)$ (of length $3$) defined by:
 
 $$\begin{array}{rcl}
 {\go g}_{-1}=V^*&=& \{\begin{pmatrix}0&0\\
 Y&0
\end{pmatrix}, Y\in M_{n}(\C)\}\\

{\go g}_{0}&=& \{\begin{pmatrix}A&0\\
 0&B
\end{pmatrix}, A,B \in M_{n}(\C),\, Tr(A+B)=0\}\\
  
 {\go g}_{1}=V&=& \{\begin{pmatrix}0&X\\
 0&0
\end{pmatrix}, X\in M_{n}(\C)\}\\
 \end{array}$$
 
 
 As an invariant form on ${\go {sl}}_{2n}(\C)$ we will take $B(\alpha,\beta)= Tr(\alpha\beta), (\alpha,\beta\in {\go {sl}}_{2n}(\C)$). This is just a multiple of the Killing form. Let us call $B_{0}$ the restriction of $B$ to ${\go g}_{0}$. The form $B_{0}$ is of course nondegenerate.  The representation $({\go g}_{0}, \rho, V)$ is defined by the bracket. Therefore   we have $$\rho(\begin{pmatrix}A&0\\
 0&B
\end{pmatrix})(\begin{pmatrix}0&X\\
 0&0
\end{pmatrix}) =\begin{pmatrix}0&AX-XB \\
 0&0
\end{pmatrix}
.$$
\vskip 5pt

If we consider ${\go {sl}}_{2n}(\C)$ as the local Lie algebra $\Gamma(\go g_{0}, B_{0}, \rho)$ we now from Proposition \ref{prop-reductive=minimal} that $\go{g}_{min}( \Gamma({\go g}_{0},B_{0},\rho))={\go {sl}}_{2n}(\C)$.

Next we will modify the form $B_{0}$ in the following manner. Let $\lambda=(\lambda_{1},\lambda_{2})\in (\C^*)^2$ and set $B_{0}^\lambda(\begin{pmatrix}A&0\\
 0&B
\end{pmatrix},\begin{pmatrix}A'&0\\
 0&B'
\end{pmatrix})=\lambda_{1}Tr(AA')+\lambda_{2}Tr(BB')$.  

An easy computations shows that $B_{0}^\lambda$ is non degenerate if and only if $\lambda_{1}\neq0$, $\lambda_{2}\neq 0$ and $\lambda_{1}+\lambda_{2}\neq0$.

It can also be shown (see Lemma \ref{lemma-forme-locale}) that $B_{0}^\lambda$ extends to an invariant non degenerate symmetric bilinear form on $\Gamma({\go g}_{0},B_0^\lambda,\rho)$. If $\lambda_{1}\neq\lambda_{2}$, $B_{0}^\lambda$ and $B_{0} $ are not proportional. Therefore the algebra $\go{g}_{min}( \Gamma({\go g}_{0},B_{0}^\lambda,\rho))$ is not isomorphic to ${\go {sl}}_{2n}(\C)=  \Gamma({\go g}_{0},B_{0},\rho)$ (uniqueness of the Killing form up to constants). In fact, using  Proposition \ref{prop-finie-ss} below, one can prove that actually $\go{g}_{min}( \Gamma({\go g}_{0},B_{0}^\lambda,\rho))$ is infinite dimensional if $\lambda_{1}\neq\lambda_{2}$

 \end{example}

\vskip 5pt
\subsection { Transitivity} \hfill
 \vskip5pt

  Let us  recall the notion of transitivity introduced by V. Kac.
 
 \begin{definition}\label{def-transitivity} \ {\rm (Kac \cite{Kac1}, Definition 2)}
 
 Let ${\go g}$ $($resp. $\Gamma$$)$ be a graded Lie algebra $($resp. a local Lie algebra$)$. Then ${\go g}$ $($resp. $\Gamma$$)$ is said to be transitive if 
 
 - for $x\in {\go g}_{i}$, $i\geq 0$, $[x,{\go g}_{-1}]=\{0\}\Rightarrow x=0$
 
  - for $x\in {\go g}_{i}$, $i\leq 0$, $[x,{\go g}_{1}]=\{0\}\Rightarrow x=0$.
 \end{definition}
 
 In particular if ${\go g}$ (or $\Gamma$) is transitive, then the modules $({\go g}_{0},{\go g}_{-1})$ and $({\go g}_{0},{\go g}_{1})$ are faithful.
 
 \begin{rem}\label{rem-centre} \hfill
 
 It is easy to see that if  a graded Lie algebra $\go{g}$ is transitive then its center $Z(\go{g})$ is trivial.
 \end{rem}
 
 If $A$ is a subset of a vector space $V$, we denote by $\langle A \rangle $ the subspace of $V$ generated by $A$.
 
 \begin{prop}\label{prop-transitive} \hfill
 
   Let $({\go g}_{0},B_{0},(\rho,V))$ be a fundamental triplet. 
   
   {\rm 1)} The local Lie algebra $\Gamma({\go g}_{0},B_0,\rho)$ $($or the minimal algebra $\go{g}_{min}(\Gamma(\go{g}_{0},B_{0},\rho))$$)$ is transitive if and only if $(\rho,V)$ is faithful and $\langle {\go g}_{0}.V \rangle =V$ and $\langle {\go g}_{0}.V^* \rangle =V^*$.

{\rm 2)} If the representation $(\rho,V)$ is completely reducible, then the local Lie algebra $\Gamma({\go g}_{0},B_0,\rho)$ $($or the minimal algebra $\go{g}_{min}(\Gamma(\go{g}_{0},B_{0},\rho))$$)$ is transitive if and only if $(\rho,V)$ is faithful and $V$ does not contain the trivial module.

 \end{prop}
 \begin{proof}\hfill

 As a minimal graded Lie algebra with a transitive local part is transitive (\cite{Kac1}, Prop. 5 page 1278), it is enough to prove the proposition for the local part.
 
 1)  Suppose that $\Gamma({\go g}_{0},B_0,\rho)$ is transitive. We have already remarked that then the representation $(\rho,V)$ is faithful (and hence  $(\rho^* ,V^*)$ is  faithful too).  If $\langle {\go g}_{0}.V \rangle \neq V$, then there exists $Y\in V^*, Y\neq0$ such that $Y({\go g}_{0}.V)= 0$. From the definition of the bracket   we obtain that $B_{0}([V,Y],{\go g}_{0})=0$. Hence $[Y,V]=\{0\}$. This contradicts the transitivity. Similarly one proves that transitivity implies $\langle {\go g}_{0}.V^* \rangle =V^*$.  Conversely  suppose that $(\rho,V)$ is faithful and $\langle {\go g}_{0}.V \rangle =V$ and $\langle {\go g}_{0}.V^* \rangle =V^*$. The first of these assumptions is one of the conditions needed for the transitivity. Suppose also that $[X,V^*]=\{0\}$ for an $X\in V$. Then $B_{0}([X,Y],U)=0$ for all $U\in {\go g}_{0}$ and all $Y\in V^*$. Hence $Y(U.X)=-U.Y(X)=0$. Therefore, as $\langle {\go g}_{0}.V^* \rangle =V^*$, we have  $V^*(X)=0$ and hence $X=0$. The same proof, using the identity  $\langle {\go g}_{0}.V \rangle =V$, shows that  $[Y,V]=\{0\}$ implies $Y=0$. Hence $\Gamma({\go g}_{0},B_0,\rho)$ is transitive.

 2) Let $V=\oplus_{i=1}^kV_{i}$  be  a decomposition of $V$ into irreducibles. If $V_{i}$  is  not  the trivial module we have of course $\langle \go{g}_{0}.V_{i}\rangle =V_{i}$. It is then easy to see that the preceding condition  "$\langle {\go g}_{0}.V \rangle =V$ and $\langle {\go g}_{0}.V^* \rangle =V^*$"  is equivalent to the condition "$V$ does not contain the trivial module".  

 \end{proof}
 
 
 
 
 \begin{rem}\label{rem-grading-elem-center}\hfill
 
 1) Suppose that there exists an element $H\in {\go g}_{0}$ such that $H.X=X$ for all $X\in V$. Such an element is called {\it grading element}. Then obviously the conditions  $\langle {\go g}_{0}.V \rangle =V$ and $\langle {\go g}_{0}.V^* \rangle =V^*$ are satisfied. In this case the local Lie algebra $\Gamma({\go g}_{0},B_0,V)$ is transitive if an only if $(\rho,V)$ is faithful.

 
 2) Suppose that the representation $(\rho,V)$ is faithful and completely reducible.  Let $V=\oplus_{i=1}^kV_{i}$  be  a decomposition of $V$ into irreducibles. Then by Schur's Lemma we obtain that $\dim Z(\go{g}_{0})\leq k$ ($Z(\go{g}_{0})$ denotes the center of $\go{g}_{0}$). Hence if the local Lie algebra $\Gamma({\go g}_{0},B_0,\rho)$ is transitive, then $\dim Z(\go{g}_{0})\leq k$. In particular if $V$ is irreducible, then  $\dim Z(\go{g}_{0})\leq 1$.

 \end{rem}
 The next proposition describes the structure of the algebra $\go{g}_{min}(\Gamma({\go g}_{0},B_0,\rho))$ in the non transitive case, under some assumptions.

 \begin{prop}\label{prop-cas-non-transitif} \hfill
 
 Let $\go{g}_{0}$ be a reductive Lie algebra, and let $B_{0}$ be a non degenerate invariant symmetric bilinear  on $\go{g}_{0}$. Let $(\rho,V)$ be a finite dimensional completely reducible representation. Denote by $\go{g}_{0}^k$ the kernel of the representation. We suppose also that the restriction of $B_{0}$ to $Z(\go{g}_{0}^k)=Z(\go{g}_{0})\cap \go{g}_{0}^k$ is non-degenerate. Then if we denote  by $\go{g}_{0}^f$ the ideal of $\go{g}_{0}$ orthogonal to $\go{g}_{0}^k$ we have $\go{g}_{0}=\go{g}_{0}^k\oplus \go{g}_{0}^f$.  Denote also by $V_{0}$ the isotypic component  of the trivial module in $V$ and by $V_{1}$ the $\go{g}_{0}$-invariant complementary subspace   to $V_{0}$ $(V=V_{1}\oplus V_{0}).$
 
 Then:
 
 $1)$  $a)$  $$\go{g}_{max}(\Gamma({\go g}_{0},B_0,V))\simeq \go{g}_{max}(\Gamma(\go{g}_{0}^f,{B_{0}}_{|_{\go{g}_{0}^f}},V))\times  \go{g}_{0}^k \hskip 8pt \text{ and } $$
 $b)$ 
 $$\go{g}_{min}(\Gamma({\go g}_{0},B_0,V))\simeq \go{g}_{min}(\Gamma(\go{g}_{0}^f,{B_{0}}_{|_{\go{g}_{0}^f}},V))\times  \go{g}_{0}^k.$$
 
$2)$ Moreover
$$\go{g}_{min}(\Gamma({\go g}_{0},B_0,V))\simeq \go{g}_{min}(\Gamma(\go{g}_{0}^f,{B_{0}}_{|_{\go{g}_{0}^f}},V_{1}))\times  (V_{0}\oplus V_{0}^*)\times \go{g}_{0}^k$$
$($$\go{g}_{min}(\Gamma(\go{g}_{0}^f,{B_{0}}_{|_{\go{g}_{0}^f}},V_{1}))$ is transitive and   $[V_{0},V_{0}]=[V_{0}^*,V_{0}^*]=[V_{0},V_{0}^*]=\{0\}$$)$.

$3)$ $$Z(\go{g}_{min}(\Gamma({\go g}_{0},B_0,V)))=(V_{0}\oplus V_{0}^*)\times Z(\go{g}_{0}^k)$$
 \end{prop}
 
 \begin{proof}\hfill
 
 1) From the assumptions we obtain that we have the following isomorphism of local Lie algebras:
 $$\Gamma({\go g}_{0},B_0,V)\simeq \Gamma(\go{g}_{0}^f,{B_{0}}_{|_{\go{g}_{0}^f}},V)\times  \go{g}_{0}^k.$$
 Then Remark \ref{rem-decomposition} implies $a)$ and $b)$.
  
  2)   Proposition \ref{prop-transitive} implies that   $\go{g}_{min}(\Gamma(\go{g}_{0}^f,{B_{0}}_{|_{\go{g}_{0}^f}},V_{1}))$ is transitive.

 It is easy to see that
  $$\Gamma({\go g}_{0},B_0,V)=\Gamma(\go{g}_{0}^f,{B_{0}}_{|_{\go{g}_{0}^f}}V_{1})\times (V_{0}\oplus V_{0}^*)\times  \go{g}_{0}^k.$$
 This is a direct sum of local Lie algebras where $V_{0}\oplus V_{0}^*$ and $\go{g}_{0}^k$ are already Lie algebras.
 Again Remark \ref{rem-decomposition} implies 
 $$\go{g}_{min}(\Gamma({\go g}_{0},B_0,V))\simeq \go{g}_{min}(\Gamma(\go{g}_{0}^f,{B_{0}}_{|_{\go{g}_{0}^f}}V_{1}))\times (V_{0}\oplus V_{0}^*)\times  \go{g}_{0}^k.$$
 
 3) As the center of $\go{g}_{min}(\Gamma({\go g}_{0},B_0,V))$ is trivial (Remark \ref{rem-centre}), the third assertion is now clear.

\end{proof}

 \begin{prop}\label{prop-finie-ss}\hfill
 
 Let $(\go{g}_{0},B_{0},\rho)$ be a fundamental triplet where $\go{g}_{0}$ is reductive and where $\rho$ is completely reducible.   Suppose that the local Lie algebra $\Gamma({\go g}_{0},B_0,\rho)$ is transitive. Then if $\dim(\go{g}_{min}(\Gamma({\go g}_{0},B_0,\rho)))$ is finite,  the Lie algebra $\go{g}_{min}(\Gamma({\go g}_{0},B_0,\rho))$ is semi-simple. 
 
  \end{prop}
 
 \begin{proof}\hfill
 
   It suffices to prove that if $\Gamma({\go g}_{0},B_0,\rho)$ cannot be  decomposed as in Remark \ref{rem-decomposition}, then $\go{g}_{min}(\Gamma({\go g}_{0},B_0,\rho))$ is simple. We know from \cite{Kac1} (Prop. 5, p.1278) that $\go{g}_{min}(\Gamma({\go g}_{0},B_0,\rho))$ is transitive. Denote by $\go{g}_{min}(\Gamma({\go g}_{0},B_0,\rho))=\oplus_{i\in\Z} \go{g}_{i}$ the grading. Let $\go{a}$ be a non zero  ideal of $\go{g}_{min}(\Gamma({\go g}_{0},B_0,\rho))$.    Let $a\in \go{a}$, $a\neq0$. Let $a=a_{-i}+ a_{-i+1}+\dots+a_{j}$ be the decomposition of $a$ according to the grading of $\go{g}_{min}(\Gamma({\go g}_{0},B_0,\rho))$, where either $-i\leq 0$ and $a_{-i}\neq 0$, or $j\geq 0$, and $a_{j}\neq 0$. Suppose for example that $-i\leq 0$ and $a_{-i}\neq 0$. From the transitivity of $\go{g}_{min}(\Gamma({\go g}_{0},B_0,\rho))$, we know that there exists $x_{1}^{i}\in \go{g}_{1}$ such that $[a_{-i}, x_{1}^{i}]\neq0$. Then $[a,x_{1}^{i}]=[a_{-i},x_{1}^{i}]+\dots +[a_{j},x_{1}^{i}]\in \go{a}$ therefore there exists an element $a'\in \go{a}$, such that $a'=a'_{-i+1}+ \dots+a'_{j+1}$ ($a'_{k}\in \go{g}_{k}$) and $a'_{-i+1}\neq 0$. By induction we prove that there exists an element $x= x_{0}+x_{1}+\dots \in \go{a}$ such that $x_{0}\neq0$ and also an element   $y=y_{1}+\dots \in \go{a}$ such that $y_{1}\neq 0$. Let $\cal{T}_{k}=\oplus_{n\geq k}\go{g}_{i}$ ($k\geq 0$). Denote by $\widetilde{\go{a}}_{0}$ (resp. $\widetilde{\go{a}}_{1}$) the projection of $\go{a}\cap {\cal T}_{0}$ on $\go{g}_{0}$ (resp.  the projection of $\go{a}\cap {\cal T}_{1}$ on $\go{g}_{1}$). 
   
   The preceding considerations show that $\widetilde{\go{a}}_{0}\neq \{0\}$ and $\widetilde{\go{a}}_{1}\neq \{0\}$. As $\go{a}$ is an ideal, $\widetilde{\go{a}}_{0}$ is an ideal of $\go{g}_{0}$ and $\widetilde{\go{a}}_{1}$ is a sub-$\go{g}_{0}$-module of $\go{g}_{1}=V$. Let $\widetilde{\go{b}}_{0}$ be the orthogonal of $\widetilde{\go{a}}_{0}$ in $\go{g}_{0}$ with respect to $B_{0}$, and let $\widetilde{\go{b}}_{1}$ be a $\go{g}_{0}$-invariant complementary space to $\widetilde{\go{a}}_{1}$ in $\go{g}_{1}$. That is $\go{g}_{1}=\widetilde{\go{a}}_{1}\oplus \widetilde{\go{b}}_{1}$.   As $\go{a}$ is an ideal we obtain $[\widetilde{\go{a}}_{0},\widetilde{\go{b}}_{1}]=\{0\}$. Let now $B$ be the extended form as defined in Proposition \ref{prop-forme-globale} below. Then,  as $[\widetilde{\go{a}}_{1},\go{g}_{-1}]\subset \widetilde{ \go{a}}_{0}$, we have for all $Y\in \go{g}_{-1}$, $B([\widetilde{\go{b}}_{0},\widetilde{\go{a}}_{1}],Y)=B(\widetilde{\go{b}}_{0},[\widetilde{\go{a}}_{1},Y])=\{0\}$. This shows that $[\widetilde{\go{b}}_{0},\widetilde{\go{a}}_{1}]$ is orthogonal to $\go{g}_{-1}$. Therefore $[\widetilde{\go{b}}_{0},\widetilde{\go{a}}_{1}]=\{0\}$. Let $x\in \widetilde{\go{a}}_{0}\cap \widetilde{\go{b}}_{0}$. Then $[x, \widetilde{\go{a}}_{1}+\widetilde{\go{b}}_{1}]=[x,\go{g}_{1}]=\{0\}$. As $\Gamma(\go{g}_{0}, B_{0}, \rho)$ is transitive we obtain that $x=0$. Hence $\go{g}_{0}=\widetilde{\go{a}}_{0}\oplus \widetilde{\go{b}}_{0}$.

   We have supposed that  $\Gamma({\go g}_{0},B_0,\rho)$ is not decomposable  in the sense of Remark \ref{rem-decomposition}. Then $\widetilde{\go{b}}_{0}=\{0\}$ and $\widetilde{\go{b}}_{1}=\{0\}$, and $\go{g}_{0}=\widetilde{\go{a}}_{0}$ and $\go{g}_{1}=\widetilde{\go{a}}_{1}$.
   
  As $\go{g}_{min}(\Gamma({\go g}_{0},B_0,\rho))$ is finite dimensional we can write $\go{g}_{min}(\Gamma({\go g}_{0},B_0,\rho))=\oplus_{i=-n}^n \go{g}_{i}$. As $\go{g}_{min}(\Gamma({\go g}_{0},B_0,\rho))$ is generated by its local part, we obtain that  any $X_{n}\in \go{g}_{n}$ is a linear combination of elements of the form   $ [\dots[X_{1}^1,X_{1}^2]\dots]X_{1}^n]  $ where $X_{1}^1,\dots,X_{1}^n \in \go{g}_{1}$. But as $\go{g}_{1}=\widetilde{\go{a}}_{1}$, we obtain that $X_{n}\in \go{g}_{n}\cap \go{a}$, and hence $\go{g}_{n}\cap \go{a}=\go{g}_{n}$.
  
  From the transitivity, we know that there exist $Y_{1}^1,\dots,Y_{1}^n\in \go{g}_{-1}$ such that $[Y_{1}^n,[Y_{1}^{n-1}, \dots[Y_{1}^1,X_{n}]\dots]\neq 0$. This proves that $\go{a}_{1}=\go{a}\cap \go{g}_{1} \neq \{0\}$ and $\go{a}_{0}=\go{a}\cap \go{g}_{0} \neq \{0\}$. Then the same reasoning as above shows that $\go{a}_{0}=\go{g}_{0}$ and $\go{a}_{1}=\go{g}_{1}$. As $[\go{g}_{0},\go{g}_{-1}]=\go{g}_{1}$ (this is again the transitivity condition), we have also that $\go{a}_{-1}= \go{a}\cap \go{g}_{-1}=\go{g}_{-1}$.
  
  Finally we have proved that $\Gamma({\go g}_{0},B_0,\rho)\subset \go{a}$. Hence $\go{g}_{min}(\Gamma({\go g}_{0},B_0,\rho))=\go{a}$.
 
 \end{proof}

 \begin{cor}\label{cor-dim-centre} \hfill
 
 Suppose that $\go{g}_{0}$ is reductive and that $(\rho,V)$ is a faithful completely reducible $\go{g}_{0}$-module which does not contain the trivial module. 
 
 Let $k$ denote the number of irreducible components of $V$. 

 If $\dim Z(\go{g}_{0})< k$, then $\dim \go{g}_{min}(\Gamma(\go{g}_{0}, B_{0}, \rho))=+\infty$.
 
\end{cor}

\begin{proof}\hfill

First remember from Remark \ref{rem-grading-elem-center} $2)$  that, as  the representation is faithful, we have  $\dim Z(\go{g}_{0})\leq k$.

From Proposition \ref{prop-finie-ss} we know that if $\dim \go{g}_{min}(\Gamma(\go{g}_{0}, B_{0}, \rho))<+\infty$ then the Lie algebra $\go{g}_{min}(\Gamma(\go{g}_{0}, B_{0}, \rho))$ is semi-simple. Then, if $\go{g}_{min}(\Gamma(\go{g}_{0}, B_{0}, \rho))=\oplus_{i=-n}^{i=n}\go{g}_{i}$ is a grading, the Lie algebra  $\go{g}_{0}$ is a Levi sub-algebra of $\go{g}_{min}(\Gamma(\go{g}_{0}, B_{0}, \rho))$ (see Example \ref{ex-PH-paraboliques}). But then $\dim Z(\go{g}_{0})=k$ (see  for example   Proposition 4.4.2 d) in \cite{Rubenthaler-note-PV}).

\end{proof}
\vskip 5pt
\subsection { Invariant bilinear forms} \hfill
 \vskip5pt

  Consider again a general quadratic Lie algebra $\go{g}_{0}$ with a non degenerate symetric bilinear form  $B_{0}$. We will first show that   $B_{0}$  extends  to an invariant form on the  local Lie algebra $ \Gamma({\go g}_{0},B_0,\rho)$.  
  \goodbreak
  
   Define a symmetric bilinear form $B$ on $ \Gamma({\go g}_{0},B_0,\rho)$ by setting:
 
 $-$ $\forall\, u, v \in {\go g}_{0}, B(u,v)=B_{0}(u,v)$
 
 $-$ $\forall \, u  \in  {\go g}_{0}, \forall\, X\in {\go g}_{1}=V, \, \forall \,Y\in{\go g}_{-1}=V^*$
  $$B(u,X)=B(X,u)=B(u,Y)=B(Y,u)=0 \eqno (*)$$

  $-$ $\forall\, X\subset{\go g}_{1}=V, \, \forall \,Y\subset{\go g}_{-1}=V^*$, $B(X,Y)=B(Y,X)=Y(X) \hfill (**)$

 \begin{lemma}\label{lemma-forme-locale}\hfill

  a) The form  $B$ is a non degenerate  invariant form on $\Gamma({\go g}_{0},B_0,\rho)$ $($the definition of an invariant form on a local Lie algebra is analogous to the Lie algebra case$)$.
  
  b) Suppose that there exists a grading element in $\Gamma({\go g}_{0},B_0,\rho)$, that is an element $H_{0}\in {\go g}_{0}$ such that $[H_{0},x]=ix$ for $x\in {\go g}_{i}$, $i=-1,0,1$,  then the preceding form $B$ is the only invariant extension of $B_{0}$ to $\Gamma({\go g}_{0},B_0,\rho)$.
 \end{lemma}
 
 \begin{proof} \hfill
 
 a) Of course as   $B_{|_{{\go g}_{0}\times {\go g}_{0}}}=B_{0}$, the invariance is verified on ${\go g}_{0}$. Let $X\in {\go g}_{1}, Y\in {\go g}_{-1}, u,v\in {\go g}_{0}$. From the definition of $[X,Y]$ (see Theorem \ref{th-algebre-locale}), we have $B([X,Y],u)=Y(u.X)$. On the other hand we have  $B(X,[Y,u])=B(X,-u.Y)=-u.Y(X)=Y(u.X)$. Hence $B([X,Y],u)=B(X,[Y,u])$. We have  also $B([u,v],X)=0=B(u,v.X)=B(u,[v,X])$. Similarly $B([u,v],Y)=B(u,[v,Y])=0$.
 
 b) Suppose that $x\in {\go g}_{i}$ and $y\in {\go g}_{j}$ with $i+j\neq 0$. Then $B([H_{0},x],y)=B(ix,y)=iB(x,y)=-B(x,[H_{0},y]) =-jB(x,y)$. Therefore $(i+j)B(x,y)=0$, and hence $B(x,y)=0$.   We also have for $X\in V={\go g}_{1}$ and $Y\in V^*= {\go g}_{-1}$: $B([X,Y], H_{0})=Y(X)= B(X,[Y,H_{0}])= B(X,Y)$. Hence conditions $(*)$ and $(**)$ are satisfied.

 \end{proof}

  \begin{prop}\label{prop-forme-globale}  \hfill

 
 {\rm 1)} If $\Gamma(\go{g}_{0},B_{0},  \rho)$ is transitive, then the extended form $B$ on 
$\go{g}_{min}(\Gamma(\go{g}_{0},B_{0}, \rho))$ defined in Proposition \ref{prop-Kac} is non-degenerate.

 {\rm 2)} The extended form $B$ on $\go{g}_{min}(\Gamma(\go{g}_{0},B_{0}, \rho))$ is also  non-degenerate in the case where $\go{g}_{0}$ is reductive, the representation $V$ is completely reducible, and the restriction of $B_{0}$ to $Z(\go{g}_{0})\cap (\ker \rho)$ is non degenerate.
 
 \end{prop}  
 
 \begin{proof}\hfill

 1) Suppose  that $\Gamma(\go{g}_{0},B_{0}, \rho)$ is transitive. Let us  denote the grading by    $\go{g}_{min}(\Gamma(\go{g}_{0},B_{0},V))=\oplus_{i\in \Z} \go{g}_{i}$. We must prove that if $X\in \go{g}_{i}$ is such that $B(X,Y)=0$  for all $Y\in \go{g}_{-i}$, then $X=0$. We will first prove the result by induction for $i\geq0$.
 
 From the definition of $B$ on   $\Gamma(\go{g}_{0},B_{0}, \rho)$ we see that the result is true for $i=0$ and $i=1$. Suppose now that the result is true for $i< k$. Let $x_{k}\in \go{g}_{k}$ such that $B(x_{k}, \go{g}_{-k})=0$. Then for all $x_{-1}\in \go{g}_{-1}$ and all $x_{-k+1}\in \go{g}_{-k+1}$ we have $B(x_{k}, [x_{-1},x_{-k+1}])=0$. And hence $B([x_{k},x_{-1}],x_{-k+1})=0$. From the induction hypothesis we get $[x_{k},x_{-1}]=0$ for all $x_{-1}\in \go{g}_{-1}$.    But we know from \cite{Kac1} (Prop. 5, p. 1278) that a minimal graded Lie algebra with a transitive local part is transitive. This implies that $x_{k}=0$.
 The same proof works for $i\leq 0$.
 
 2) This is a consequence of 1) and the explicit form of   $\go{g}_{min}(\Gamma(\go{g}_{0},B_{0}, \rho))$ given in this case in Proposition \ref{prop-cas-non-transitif} 2).
   
   \end{proof}

   \begin{definition}\label{def-isometrie-alg-grad}\hfill
   
   $1)$ A quadratic graded Lie algebra is a pair $(\go{g}, B)$ where $\go{g}=\oplus_{i\in \Z}\go{g}_{i}$ is a graded Lie algebra, and where $B$ is a non-degenerate symmetric invariant bilinear form on $\go{g}$ such that $B(\go{g}_{i},\go{g}_{j})= 0$ if $i\neq -j$.

$2)$ Let $(\go{g}^1, B^1)$ and $(\go{g}^2, B^2)$ be two quadratic graded Lie algebras. An isomorphism from $(\go{g}^1, B^1)$ onto $(\go{g}^2, B^2)$ is an isomorphism of quadratic graded Lie algebras $\Psi: \go{g}^1\longrightarrow \go{g}^2$ such that:
$$\forall x,y \in \go{g}^1,\,\, B^2(\Psi(x),\Psi(y))=B^1(x,y).$$
   
 \end{definition}

 \begin{theorem}\label{th-bijection}\hfill
 
 Let ${\cal T}$ be the set  of isomorphism classes of fundamental triplets such that the representation $(\rho, V)$ is faithful and such  that $\langle \go{g}_{0}.V\rangle =V$ and $\langle  \go{g}_{0}.V^*\rangle =V^*$.
 Let ${\cal G}$ be the set of isomorphism  classes of transitive quadratic graded Lie algebras.  
 
  The map $$\tau: {\cal T}\longrightarrow {\cal G}$$ defined by $\tau(\overline{(\go{g}_{0},B_{0},\rho)})=\overline{(\go{g}_{min}(\Gamma(\go{g}_{0},B_{0}, \rho)),B)}$ is a bijection $($here $B$ is the form defined in Proposition \ref{prop-forme-globale} and  the "overline" denotes  the equivalence class$)$.
 
  

 \end{theorem}
 \begin{proof}\hfill
 
 If $T $ is a fundamental triplet satisfying the given assumptions,  then its local part $\Gamma(T)$ is transitive from Proposition \ref{prop-transitive}. And hence $\go{g}_{min}(\Gamma(T))$ is transitive from \cite{Kac1}, Prop. 5 b), p. 1278. Moreover,  with the form $B$ defined in Proposition \ref{prop-forme-globale},   the Lie algebra $\go{g}_{min}(\Gamma(T))$ becomes a quadratic graded Lie algebra.

 Let $T_{1}$ and $T_{2} $  be two fundamental triplets and let $\Gamma(T_{1})$ and $\Gamma(T_{2})$ be the corresponding local Lie algebras. Suppose that $T_{1}\simeq T_{2}$. Then from Theorem \ref{th-extension-iso}, we have $\Gamma(T_{1})\simeq \Gamma(T_{2})$, and hence $\go{g}_{min}(T_{1})\simeq \go{g}_{min}(T_{2})$ (isomorphism of graded Lie algebras).  
 
Let $T_{1}=\Gamma(\go{g}_{0}^1,B_{0}^1,\rho_{1})$, and $T_{2}=\Gamma(\go{g}_{0}^2,B_{0}^2,\rho_{2})$ the explicit triplets we consider. Let $\Psi: \go{g}_{min}(\Gamma(\go{g}_{0}^1,B_{0}^1,\rho_{1}))\longrightarrow \go{g}_{min}(\Gamma(\go{g}_{0}^2,B_{0}^2,\rho_{2}))$ be the preceding isomorphism. We will   prove that $\Psi$ is in fact an isomorphism of quadratic graded Lie algebras. Therefore we must prove that if $x\in \go{g}_{min}(\Gamma(\go{g}_{0}^1,B_{0}^1,\rho_{1}))_{i}$ $(i\geq 0)$) and $y\in \go{g}_{min}(\Gamma(\go{g}_{0}^1,B_{0}^1,\rho_{1}))_{-i}$, then 
$$B^2(\Psi(x),\Psi(y))=B^1(x,y)\eqno{(*)}$$
where $B^1$ and  $B^2$ are the extended forms defined in Proposition \ref{prop-forme-globale}. 
We will prove $(*)$ by induction on $i$. It is clear that $(*)$ is true for $i=0$ (see condition $a) $ in definition \ref{def-morphism-triplets}). We need also to prove this for $i=1$. Set $\Psi_{|_{V_{1}}}=\gamma$. Then from the proof of Theorem \ref{th-extension-iso} we have $\Psi_{|_{V_{-1}}}= ^t\kern -4pt \gamma^{-1}$. Therefore for $x\in V_{1}, y\in V_{1}^*$, we have:
$$B^2(\Psi(x),\Psi(y))=B^2(\gamma(x), ^t\kern -4pt \gamma^{-1}(y))= ^t\kern -4pt \gamma^{-1}(y)(\gamma(x))=y(x)=B^1(x,y)$$
This proves $(*)$ for $i=1$.

Suppose now that $(*)$ is true for $0\leq i< k$. In the rest of the proof  the elements $x_{i}$ or $ y_{i}$ belong always to $ \go{g}_{min}(\Gamma(\go{g}_{0}^1,B_{0}^1,\rho_{1}))_{i}$.  Any $x\in  \go{g}_{min}(\Gamma(\go{g}_{0}^1,B_{0}^1, \rho_{1}))_{k}$ is a linear combination of elements of the form $[x_{k-1},x_{1}]$ and any $y\in \go{g}_{min}(\Gamma(\go{g}_{0}^1,B_{0}^1,\rho_{1}))_{-k}$ is a linear combination of elements of the form $[y_{-k+1}, y_{-1}]$. Then from the proof of Proposition 7 p. 1279 of \cite{Kac1} the extended form is defined inductively by
$$B^1([x_{k-1},x_{1}],[y_{-k+1},y_{-1}])=B^1([[x_{k-1},x_{1}],y_{-k+1}], y_{-1})$$
and the same is true for $B^2$.

Therefore 
$$\begin{array}{rl}
B^2(\Psi([x_{k-1},x_{1}]),\Psi([y_{-k+1},y_{-1}])&=B^2( [\Psi(x_{k-1}),\Psi(x_{1})], [\Psi(y_{-k+1}),\Psi(y_{-1})])\\
&=B^2([[\Psi(x_{k-1}),\Psi(x_{1})],\Psi(y_{-k+1})], \Psi(y_{-1}))\\
&=B^2(\Psi([[ x_{k-1},x_{1}], y_{-k+1}]), \Psi(y_{-1}))\\
\text{(by induction:)}&=B^1([[ x_{k-1},x_{1}], y_{-k+1}], y_{-1})\\
&= B^1([x_{k-1},x_{1}],[y_{-k},y_{-1}])
\end{array}$$ 
Therefore $(*)$ is proved.  Hence if $T_{1}\simeq T_{2}$, then the algebras  $\go{g}_{min}(T_{1})$and $ \go{g}_{min}(T_{2})$ are isomorphic as quadratic graded Lie algebras. In other words the map $\tau$ is well defined.

Let ($\go{g}$,B) be a transitive quadratic graded Lie algebra, whith local part $\Gamma(\go{g})= \go{g}_{-1}\oplus \go{g}_{0}\oplus \go{g}_{1}$.  As a transitive graded algebra is minimal (\cite{Kac1}, Prop. 5 a), p. 1278) and has of course a transitive local part, we have $\go{g}=\go{g}_{min}(\Gamma(\go{g}))$. Moreover the properties of $B$ imply immediately that if $V=\go{g}_{1}$ denotes the corresponding $\go{g}_{0}$-module then $\go{g}_{-1}=V^*$, the dual $\go{g}_{0}$-module. And then, from the transitivity we obtain that  the representation $(\rho, V)$ is faithful and   that $\langle \go{g}_{0}.V\rangle =V$ and $\langle  \go{g}_{0}.V^*\rangle =V^*$ (Proposition \ref{prop-transitive}). This proves that the map $\tau$ is surjective.

Suppose now that $\Psi: \go{g}_{min}(\Gamma(\go{g}_{0}^1,B_{0}^1,\rho_{1}))\longrightarrow \go{g}_{min}(\Gamma(\go{g}_{0}^2,B_{0}^2,\rho_{2}))$ is now an isomorphism of {\it quadratic} graded Lie algebra. Set $\Psi_{|_{\go{g}_{0}^0}}=A$ and $\Psi_{|_{\go{g}_{0}^1}}=\gamma$. Then $A$ is a Lie algebra isomorphism from $\go{g}_{0}^1$ onto $\go{g}_{0}^2$ and $\gamma$ is an isomorphism fron $\go{g}_{1}^1$ onto $\go{g}_{1}^2$ which satisfy:

-- $\forall U, V\in \go{g}_{0}^1$, $B^2(A(U),A(V))=B^1(U,V)$ and 

-- $\forall U, \in \go{g}_{0}^1$ and $X\in \go{g}_{1}^1$, $\Psi([U,X])=[\Psi(U),\Psi(X)]=[A(U),\gamma(X)]$ and this means exactly, that in the usual notations, $\gamma\circ\rho_{1}(U)X= \rho_{2}(A(U))\circ\gamma(X)$.

Hence $(A,\gamma)$ is an isomorphism of fundamental triplets (see Definition \ref{def-morphism-triplets}).

This proves that the map $\tau$ is injective.

 \end{proof}

\vskip 20pt

\section{ ${\go {sl}}_{2}$-triples}


\vskip 5pt
In this section we will suppose that the pair  $(\go{g}_{0}, \rho)$ satisfies the following assumption. 

 \vskip 5pt 
 
{ \bf Assumption  (H)}:
 
 a) The Lie algebra ${\go g}_{0}$ is reductive with a one dimensional center: ${\go g}_{0}=Z({\go g}_{0})\oplus {\go g}_{0}'$ where ${\go g}_{0}'=[{\go g}_{0},{\go g}_{0}]$ and $\dim Z({\go g}_{0})=1$.  
 
 b) We suppose also that       $Z({\go g}_{0})$ acts by a non trivial character (i.e. $\rho(Z(\go{g}_{0}))=\C Id_{V}$).  
 
 \vskip 10pt 
 
 Then there exists $H_{0}\in Z({\go g}_{0})$ such that $\rho(H_{0})=2{\rm Id}_{{V}}$ (and $\rho^*(H_{0})=-2{\rm Id}_{{V^*}}$).
 
Recall also that in a Lie algebra, or in a local Lie algebra, a triple of elements $(y,h,x)$ is called an ${\go {sl}}_{2}$-triple if $[h,x]=2x$, $[h,y]=-2y$ and $[y,x]=h$.

\subsection{Associated ${\go {sl}}_{2}$-triple}

\begin{definition}\label{def-sl2-associe}   We say that the local Lie algebra $\Gamma(\go{g}_{0},B_{0}, \rho)$, or the graded Lie algebra $\go{g}_{min}(\Gamma(\go{g}_{0},B_{0}, \rho))$,  is  associated to an ${\go {sl}}_{2}$-triple if there exists $X\in V $, $Y\in V^* $  such that $(Y,H_{0},X)$ is an ${\go {sl}}_{2}$-triple.
\end{definition}

\begin{theorem} \label{th-CNS-sl2}\hfill

 The local Lie algebra $\Gamma(\go{g}_{0},B_{0}, \rho)$ is  associated  to an ${\go {sl}}_{2}$-triple if and only if there exists $X\in V$ such that $X\notin {\go g}_{0}'.X$ where ${\go g}_{0}'.X=\{U.X,\, U\in {\go g}_{0}'\}$. The set $\{X\in V,\,X\notin {\go g}_{0}'.X\}$ is exactly the set of elements in $V$ which belong to an associated ${\go {sl}}_{2}$-triple.
\end{theorem}

\begin{proof}\hfill

Recall from Lemma \ref{lemma-forme-locale} that the form $B_{0}$ extends to an invariant form $B$ on $\Gamma(\go{g}_{0},B_{0}, \rho)$.

Suppose that $\Gamma(\go{g}_{0},B_{0}, \rho)$ has an associated ${\go {sl}}_{2}$-triple $(Y, H_{0},X)$. Then $Y({\go g}_{0}'.X)= B(Y, {\go g}_{0}'.X)=B(Y, [{\go g}_{0}',X])=B([Y,X],{\go g}_{0}')=B(H_{0},{\go g}_{0}')=\{0\}$. Hence the form $Y$ is zero on ${\go g}_{0}'.X$. On the other hand $B(Y,X)=Y(X)=\frac{1}{2}B(Y,[H_{0},X])=-\frac{1}{2}B([Y,X],H_{0})=-\frac{1}{2}B(H_{0},H_{0})=-\frac{1}{2}B_{0}(H_{0},H_{0})\neq 0$. Therefore $X\notin {\go g}_{0}'.X$.

Conversely suppose that $X\notin {\go g}_{0}'.X$. We choose $Y\in V^*$ such that   $Y({\go g}_{0}'.X)=\{0\}$ and $Y(X)\neq 0$. Then $B([Y,X],{\go g}_{0}')=B(Y,{\go g}_{0}'.X)=Y({\go g}_{0}'.X)=\{0\}$. Therefore $[Y,X]\in ({\go g}_{0}')^{\perp}=\C H_{0}$. Set $[Y,X]=\lambda H_{0}$. We have also $Y(X)=B(Y,X)=\frac{1}{2}B(Y,[H_{0},X])=-\frac{1}{2}B([Y,X],H_{0})=-\frac{1}{2}B(\lambda H_{0},H_{0})=-\frac{1}{2}\lambda B( H_{0},H_{0})\neq 0$. Hence $\lambda\neq 0$. Define $\widetilde Y=\frac{1}{\lambda}Y$. Then $(\widetilde Y, H_{0},X)$ is an ${\go {sl}}_{2}$-triple.

\end{proof}

\begin{rem} \label{rem-sl2-ind-forme}\hfill

a$)$  The existence of a  ${\go {sl}}_{2}$-triple associated to the local Lie algebra $\Gamma(\go{g}_{0},B_{0}, \rho)$ does not depend on the invariant form $B_{0}$ on ${\go g}_{0}$, but only on the representation $(\go{g}_{0}',\rho_{|_{\go{g}_{0}'}},V)$.

b)  Suppose that the  module   $(\rho,V)$ satisfies  the following property:
 $${\bf (P)}\hskip 5pt  \text{ There exists } X\in V \text { such that } X\notin {\go g}_{0}'.X.$$
  Then the dual module $(\rho^*, V^*)$ satisfies the same property. This is due to the fact that one can always suppose that the center acts non trivially, and then observe that the existence of an associate $\go{sl}_{2}$-triple is a symmetric condition.

\end{rem}




\vskip 5pt

\subsection{Property {\bf (P)}, relative invariants and ${\go {sl}}_{2}$-triples}\hfill 
\vskip 5pt


\begin{definition} 
  Let $(G_{0},\pi, V)$ be a finite dimensional representation of a connected complex algebraic group $G_{0}$.  Let $X\in V$ and consider its $G_{0}$-orbit ${\cal O}_{X}=G_{0}.X$.  Then ${\cal O}_{X}$ is open in its closure $\overline{ {\cal O}_{X}}$ which is an irreducible affine variety. Let $R$ be an element of the field $\C(\overline{ {\cal O}_{X}}) $ of rational functions on   ${\cal O}_{X}$. The function $R$ is called a relative invariant on ${\cal O}_{X}$ if there exists   a rational character $\chi$ of $G_{0}$ such that
  $$\forall x \in {\cal O}_{X},   \,\,\,R(\pi(g)x)=\chi(g)R(x) \eqno( 4-2-1) $$
  \end{definition}
  \vskip 10pt
  
  In the sequel we will  suppose that the representation $(\rho,V)$ of $\go{g}_{0}$ lifts to a representation $\pi$ of a connected algebraic  group $G_{0}$ whose Lie algebra is $\go{g}_{0}$ (in other words $\rho$ is the derived representation $d\pi$).

  For convenience we will often write $g.x$ ($g\in G_{0}, x\in V$) instead of $\pi(g)x$ and also $A.x$ ($A\in \go{g}_{0}, x\in V$) instead of $\rho(A)x$.
  \vskip 5pt
  
 We will now prove   the existence of an associated ${\go {sl}}_{2}$-triple $(X,H_{0},Y)$   is equivalent to the existence of a non-trivial  relative invariant on the $G_{0}$-orbit of $X$. 
 \vskip 5pt

 It will be easier to specialize the property {\bf (P)} at a point $x$. For $x\in V$, the property ${\bf (P)}_{x}$ is defined as follows:
 
$${\bf (P)}_{x}: x\notin {\go g}_{0}'.x$$

   \vskip 10pt
   
   \begin{lemma}\label{lem-P}  \footnote{This lemma was first communicated to me by Michel Brion, the  proof given here is due to the referee.} \hfill
  
  Let $V$ be a representation of an algebraic group $H$, and let $\go{h}=\text{Lie}(H)$. Then for $x\in V$, the orbit $H.x$ is a cone if and only if $x\in \go{h}.x$.   
 \end{lemma} 
 
 \begin{proof}\hfill

   Set $N_{ x }= \{h\in H, h.x\in  \C x\}$, $C_{x}=\{h\in H, h.x=x\}$, $\go{n}_{x}=\text{Lie}(N_{x})$ and $\go{c}_{x}=\text{Lie}(C_{x})$. Considering the orbital map $h\longmapsto h.x$ from $N_{x}$ to $\C x$ we obtain
   $$\C^*x\subset H.x\Leftrightarrow \dim N_{x}/C_{x}=1\Leftrightarrow  \dim \go{n}_{x}/\go{c}_{x}=1\Leftrightarrow \C x\subset \go{h}.x$$ 
 
 \end{proof}

 \begin{theorem}\label{th-equivalence-invariants-(P)-sl2}\hfill
 
  
  Let $G_{0}'= [G_{0},G_{0}]$ be the commutator subgroup of $G_{0}$. Let $X\in V$. The following three conditions are equivalent:
  
  $1)$ ${\bf (P)}_{X}: X \notin \go{g}_{0}'.X$.
  
  $2)$  The orbit $G_{0}'.X$ is non conical.
  
  $3)$ There exists a non trivial relative invariant on the $G_{0}$-orbit ${\cal O}_{X}=G_{0}.X$.
  
  $4)$  $X$ belongs to an associated $\go{sl}_{2}$-triple $(X, H_{0}, Y)\in \go{g}_{min}(\Gamma(\go{g}_{0},B_{0}, \rho))$.
  
  \end{theorem}

  \begin{proof}\hfill
  
    The equivalence of $1)$ and $2)$ is Lemma \ref{lem-P}.
  The equivalence of $1)$ and $4)$ has already been proved in Theorem \ref{th-CNS-sl2}.  

 Assume that condition $3)$ holds. Let $R$ a non trivial relative invariant on ${\cal O}_{X}$. Let $t_{1}=R(X)$ and $t_{2}$ be two distinct values taken by $R$. Define  ${\cal O}_{i}=\{x\in {\cal O}_{X}\,|\, R(x)=t_{i}\}$ for  $i=1,2$. Then  ${\cal O}_{1}$ and ${\cal O}_{2}$ are two $G_{0}'$-stable subsets of ${\cal O}_{X}$ such that ${\cal O}_{1}\cap {\cal O}_{2}=\emptyset$. Hence  $G_{0}'.X\neq G_{0}.X$. From Lemma \ref{lem-P}  we obtain that $3) \Longrightarrow 2)$.

 Conversely, let us assume that ${\bf (P)}_{X}$ holds. Then from Lemma \ref{lem-P} we know that the $G_{0}$-orbit ${\cal O}_{X}$ splits into several $G_{0}'$-orbits. 
 
 
 



 Suppose that one of these, say $G_{0}'.v$, is open in ${\cal O}_{X}$.  Denote by $G_{0_{v}}$ be the stabilizer of $v$ in $G_{0}$ and by $G_{0_{v}}'=G_{0_{v}}\cap G_{0}'$ the stabilizer of $v$ in $G_{0}'$. Then the subgroup $G_{0}'.G_{0_{v}}$ is open in $G_{0}$ as the inverse image of $G_{0}'.v$ under  the orbital map $g\longmapsto g.v$. As $G_{0}$ is connected we have $G_{0}=G_{0}'.G_{0_{v}}$. Then ${\cal O}_{X}=G_{0}.v=G_{0}'.G_{0_{v}}.v=G_{0}'.v$ and this is not possible as we should have several $G_{0}'$-orbits in ${\cal O}_{X}$. Hence if we assume that ${\bf (P)}_{X}$ holds,  there is no open $G_{0}'$-orbits in ${\cal O}_{X}$. In particular $\dim {\cal O}_{X}-\dim G_{0}'.X > 0$. Therefore
$$\begin{array}{rl}
\dim {\cal O}_{X}-\dim G_{0}'.X&= \dim G_{0}-\dim G_{0_{X}}-\dim G_{0}'.X\\
{}&= \dim G_{0} -\dim G_{0_{X}}-\dim G_{0}'+ \dim {G_{0_{X}}'}\\
{}&= \dim (G_{0}/G_{0}')-\dim (G_{0_{X}} /{G_{0_{X}}'})\\
{}&=1-\dim (G_{0_{X}} /{G_{0_{X}}'})>0
 
\end{array}$$

This implies that $\dim (G_{0_{X}} /{G_{0_{X}}'})=0$ (hence at the Lie algebra level we have $\go{g}_{0_{X}}=\go{g}_{0_{X}}'$). Then $\det (\pi(G_{0_{X}}))$  is a finite group. Therefore there exists $p\in \N^*$ such that $\det^p({\pi(G_{0_{X}})})=1$, but from assumption ${\bf (H)}$ we know that the character $g\longmapsto \det^p(\pi(g))$ is non trivial. Then the function $R: {\cal O}_{X}\longrightarrow \C$ defined by
$$\forall g\in G_{0}\,\,\,,R(\pi(g)X)={\det}^p(\pi(g))$$
is a non-trivial relative invariant on ${\cal O}_{X}$, and this is condition $3)$.

\end{proof}





\begin{rem} 

Let us recall that a prehomogeneous vector space $(G,V)$ is said to be regular if there exists a relative invariant $P$ such that the corresponding Hessian matrix $H_{P}$ is generically non degenerate (for the general theory of prehomogeneous vector space due to M. Sato see \cite{Sa-Ki} or \cite{Rubenthaler-book}). It is  known that for  irreducible prehomogeneous vector spaces of parabolic type the regularity is equivalent to the existence of a non-trivial relative invariant on the open orbit and also  to  the existence of an associated ${\go {sl}}_{2}$-triple    (see \cite{Rubenthaler-book} (Corollaire 4.3.3. p.134) or \cite{Rubenthaler-note-PV} (Corollaire 1)). But in the preceding Theorem the representation does not need to be irreducible and there is no assumption of prehomogeneity. It works for any  representation of $\go{g}_{0}'$.

 \end{rem}
 \vskip 5pt

\begin{cor}\label{cor-P-PHparaboliques}\hfill 

Consider the particular case where $\go{g}=\displaystyle\oplus _{i=-n}^{i=n}\go{g}_{i}$ is   a  grading of a semi-simple $($finite dimensional$)$ Lie algebra  $\go{g}$ $($see example \ref{ex-PH-paraboliques}$)$ such that the representation $(\go{g}_{0}, \go{g}_{1})$ is irreducible. Then  the representation $(G_{0}, \go{g}_{1})$ is   prehomogeneous.  Denote by $\Omega$ its  open orbit. Let $x\in \go{g}_{1}$. 

If the prehomogeneous vector space is regular, then
$$x\in \go{g}_{0}'.x=[\go{g}_{0}',x] \Longleftrightarrow  x\notin \Omega \eqno{(*)}$$
If the prehomogeneous vector space is not regular, then for each $x\in \go{g}_{1}$, we have $x\in [\go{g}_{0}',x]$. Moreover except for the open orbit in the regular case, all $G_{0}$-orbits are $G_{0}'$-orbits.
\end{cor}

\begin{proof}\hfill

Suppose that $x\in \go{g}_{0}'.x=[\go{g}_{0}',x]$. This means that property $\text{\rm  non}{\bf (P)}_{x}$ holds. According to Theorem \ref{th-equivalence-invariants-(P)-sl2}, this is equivalent to the fact that $x$ does not belong to an associated $\go{sl}_{2}$-triple. From \cite{Rubenthaler-book} (Corollaire 4.3.3 p. 134), or from\cite{Rubenthaler-note-PV}, the only elements which belong to an associated $\go{sl}_{2}$-triple are those in the open orbit in the regular case. This proves the two first assertions. If  $x\in \go{g}_{1}$ is not an element of the open orbit in the regular case, then $x$ does not belong to  an associated $\go{sl}_{2}$-triple, and therefore $x\in  \go{g}_{0}'.x=[\go{g}_{0}',x] $ (Theorem \ref{th-equivalence-invariants-(P)-sl2}). Then, from Lemma \ref{lem-P}, we obtain $G_{0}'.x=G_{0}.x$.  
\end{proof}

\begin{rem}  The equivalence $(*)$ in Corollary \ref{cor-P-PHparaboliques} was first proved in the the particular case of so-called "Heisenberg gradings" of a simple Lie-algebra (over any field)  by M. Slupinski and R. Stanton  (see \cite{Slupinski-Stanton}, Lemma 3.3 p. 164). These gradings are special gradings of length 5.  It is  worth noticing that it is also possible to prove $(*)$ by combining  results of V. Kac \cite{Kac2} (more specially Lemma 1.1 p.193) and of the author (\cite{Rubenthaler-book} or \cite{Rubenthaler-note-PV}).
\end{rem}
\vskip 15pt

Let $\go{g}=\oplus_{i=-n}^n\go{g}_{i}$ be the grading corresponding to an irreducible prehomogeneous vector space of parabolic type. It is easy to see that if there exists an associated $\go{sl}_{2}$-triple, then for $x$ belonging to the open $G_{0}$-orbit $\Omega\subset \go{g}_{1}$, the map $x\longmapsto y$ (such that $(y,H_{0},x)$ is a $\go{sl}_{2}$-triple) is  $G_{0}$-equivariant from $\Omega$ to $\go{g}_{-1}\simeq \go{g}_{1}^*$. It is also known (see \cite{Rubenthaler-book} Th. 4.3.2 p.132, or \cite{Rubenthaler-note-PV}) that if $R$ is a non trivial relative invariant on $\go{g}_{1}$, then there exists a constant $c\neq0$ such that $(c \displaystyle\frac{dR(x)}{R(x)}, H_{0},x)$ is an associated  $\go{sl}_{2}$-triple. The purpose of the next two propositions is to generalize these facts.

\vskip 15pt

Let $X\in \go{g}_{1}=V$ satisfying condition $({\bf P})_{X}: X\notin \go{g}_{0}'.X$. Then any element $x$ belonging to the orbit ${\cal O}_{X}=G_{0}.X$ satisfies condition $({\bf P})_{x}$ and therefore belongs to an associated $\go{sl}_{2}$-triple. Let $T_{x}=\go{g}_{0}.x=[\go{g}_{0},x]$ denote the tangent space at $x$ to ${\cal O}_{X}$. Using the canonical isomorphism $T_{x}^*\simeq V^*/ (T_{x})^{\perp}$, any element of $T_{x}^*$ can be considered as a class  modulo $(T_{x})^{\perp}$ in $V^*=\go{g}_{-1}$.

\begin{prop}\label{prop-y=section-cotangent}  \hfill

Define, for $x\in {\cal O}_{X}$:
$$\varphi(x)=\{y\in \go{g}_{-1},\, (y,H_{0},x) \text{ is a } \go{sl}_{2}\text{-triple}\}.$$
Then $\varphi(x)$ is a class modulo $(T_{x})^{\perp}$ in $V^*=\go{g}_{-1}$. Hence the map $$x\longmapsto \varphi(x) \in   V^*/ (T_{x})^{\perp}\simeq T_{x}^*$$
is a section of the cotangent bundle $T^*({\cal O}_{X})$.

Moreover this  section is equivariant:
$$\forall g\in G_{0},\,\, \varphi(\pi(g).x)=\pi^*(g)\varphi(x).$$

\end{prop}

\begin{proof}\hfill

 Let $y_{0}\in \varphi(x)$ and let $z\in (T_{x})^{\perp}$. Then for $u\in \go{g}_{0}$, we have $B_{0}([z,x],u)=-z([u,x])=0$, as $[u,x]\in T_{x}$. Hence $y_{0}+(T_{x})^{\perp}\subset \varphi(x)$.

Conversely let $y\in \varphi(x)$. Then  $[y-y_{0},x]=0$, and therefore $B_{0}([y-y_{0}, x], u)=-(y-y_{0})([u,x])=0$ for all $u\in \go{g}_{0}$. Hence $y\in y_{0}+(T_{x})^{\perp}$. 

 Suppose that $(y,H_{0},x)$ is a  $\go{sl}_{2}$-triple. Then for $u\in \go{g}_{0}$:
$$\begin{array}{rl}
B_{0}([\pi^*(g)y,\pi(g)x],u)&=-\pi^*(g)y(u.\pi(g)x)\\
= -y(\pi(g^{-1})(u.\pi(g)x)) &=-y((\Ad g^{-1}u).x)\\
=B_{0}([y,x],\Ad g^{-1}u) &=B_{0}(H_{0},u).
\end{array}$$
Hence $[\pi^*(g)y,\pi(g)x]=H_{0}$. Therefore $\pi^*(g)y\in \varphi(\pi(g)x)$ or equivalently $\pi^*(g)\varphi(x)\subset \varphi(\pi(g)x)$ . As  $\pi^*(g) (T_{x})^{\perp}=(T_{\pi(g)x)})^{\perp}$ and $\varphi(\pi(g)x)$ is a  class modulo $(T_{x})^{\perp}$, we obtain $\pi^*(g)\varphi(x)=\varphi(\pi(g).x) $.

\end{proof}

\vskip 30pt

Let $X\in V={\go g}_{1}$ which belongs to an associated $\go{sl}_{2}$-triple. Then  the orbit  $G_{0.}X={\cal O}_{X}$ has a non trivial relative invariant $R$ by Theorem \ref{th-equivalence-invariants-(P)-sl2}. The next proposition shows how one can built an associated ${\go {sl}}_{2}$-triple containing $x\in {\cal O}_{X}$ from the knowledge of $R$. 

    For this we will now consider the  " logarithmic differential" (or "gradlog") of $R$ given by $\varphi_{R}(x)=\displaystyle \frac{dR(x)}{R(x)}\in T_{x}^*$  as a class  in $V^*/ (T_{x})^{\perp}$. In particular  $ \varphi_{R}(x)= \displaystyle\frac{dR(x)}{R(x)}$ is a subset of $V^*$.

 \begin{prop}\label{prop-invariant-sl2}\hfill
 
   Let $(\rho,V)$ be a finite dimensional representation of a semi-simple Lie algebra $\go{g}_{0}'$. Extend this representation to a fundamental triplet  $(\go{g}_{0}, B_{0}, \rho)$ satisfying assumption {\bf (H)}, and let $G_{0} $ be a connected reductive group whose Lie algebra is $\go{g}_{0}$,   on which the representation $\rho$ lifts. 
  
  Suppose that the orbit ${\cal O}_{X}=G_{0}.X$ has a non trivial relative invariant $R$ with character $\chi$. Let $B$ be the extended form on $\go{g}_{min}(\Gamma(\go{g}_{0},B_{0},V))$ defined in Proposition  \ref{prop-forme-globale}. Let $x\in {\cal O}_{X}$. Then, for any element $y$ in the  class $\varphi_{R}(x)$, $(x, H_{0}, -\displaystyle \frac{B(H_{0},H_{0})}{d\chi(H_{0})}y)$ is an associated ${\go {sl}}_{2}$-triple. Therefore, in  the notation of Proposition \ref{th-equivalence-invariants-(P)-sl2}, we have $\varphi(x)=  -\displaystyle \frac{B(H_{0},H_{0})}{d\chi(H_{0})}\varphi_{R}(x)$

  \end{prop}
 
 \begin{proof}\hfill

 


 For $A\in \go{g}_{0}$ and $x\in {\cal O}_{X}$, let us   derive the identity
 $R(\pi(\exp tA)x)=\chi(\exp tA)R(x)$ with respect to $t$, at $t=0$. We obtain
 $$dR(x)d\pi(A)x=dR(x)\rho(A)x= d\chi(A)R(x).$$
  We observe that   $B(y, [A, x])$ does only depend on the class $\varphi_{R}(x)$ of $y$. Therefore the preceding equation can be written:  
 $$B(\frac{dR(x)}{R(x)},\rho(A)x)=B(\varphi_{R}(x), [A,x])=d\chi(A) .$$
 
 
 And as $B$ is invariant we obtain:
 $$\forall A\in \go{g}_{0},\,\,\, -B([\varphi_{R}(x),x], A)=d\chi(A).$$
 As $B$ is non-degenerate and as $d\chi(\go{g}_{0}')=0$, 
 $[\varphi_{R}(x),x]$ is a fixed vector (as $x$ varies) orthogonal to $\go{g}_{0}'$. Hence $[\varphi_{R}(x),x]=cH_{0}$ $(c\in \C)$. If $A=H_{0}$, one obtains $-B([\varphi_{R}(x),x], H_{0})=d\chi(H_{0})\neq 0$ (because  $\chi$ is non trivial). Therefore   $-cB(H_{0},H_{0})=d\chi(H_{0})$ and $c= -\displaystyle\frac{d\chi(H_{0})}{B(H_{0},H_{0})}\neq 0$.  
   
 Then $(-\displaystyle \frac{B(H_{0},H_{0})}{d\chi(H_{0})}\varphi_{R}(x) ,H_{0}, x)$ is an ${\go {sl}}_{2}$-triple (this means that for $y$ in the class of $-\displaystyle \frac{B(H_{0},H_{0})}{d\chi(H_{0})}\varphi_{R}(x)$ in $V^*/ (T_{x})^{\perp}$, $(y,H_{0},x)$ is an associated $\go{sl}_{2}$-triple).
 

 \end{proof}

\begin{rem}\label{rem-equivariance-gradlog}\hfill

It is worth noticing that from the preceding result,   the "gradlog" section $x\longmapsto \displaystyle \frac{dR(x)}{R(x)}=\varphi_{R}(x)$ satisfies the same equivariance property as $\varphi$:
$$\varphi_{R}(\pi(g)x)=\pi^*(g)(\varphi_{R}(x)).$$

 



\end{rem}

\vskip 20pt

\section{ Lie algebras of polynomial type and dual pairs}

\subsection{ Lie algebras of polynomial type}\hfill

In this section we will deal with a particular kind of minimal graded Lie algebras of the form $\go{g}_{min}(\Gamma(\go{g}_{0},B_{0},\rho))$. 
 \begin{definition}\label{def-type-symplectique}\hfill 
 
 Let $W$ be a finite dimensional vector space over $\C$, and let $\go{gl}(W)$ be the Lie algebra of endomorphisms of $W$.  For $p\in \N^*$, let  $\C^p[W]$ be the vector space of homogeneous polynomials of degree $p$ on $W$.  We set $\go{g}_{0}= \go{gl}(W)$ and $V= \C^p[W]$.  For $\lambda\in \C^*$ we define the representation $\rho_{\lambda}$ of $\go{gl}(W)$ on $\C^p[W]$ by saying that ${\rho_{\lambda}}_{|_{\go{sl}(W)}}$ is the natural representation of $ \go{sl}(W)$ on $\C^p[W]$ and $\rho_{\lambda}({\rm Id}_{W})=\lambda {\rm Id}_{\C^p[W]}$. Let $B_{0}$ be a non degenerate bilinear symmetric form on $ \go{gl}(W)$.
 With  these notations, the Lie algebra of polynomial type $\go{p}^p(W,B_{0})$ is defined 
 by $$\go{p}^p(W,B_{0},\lambda)=\go{g}_{min}(\Gamma(\go{gl}(W),B_{0},\rho_{\lambda})) .$$
 
 \end{definition}
 
 \begin{rem}\label{rem-divers-symplectique}\hfill
 
  1)  Note that from Proposition \ref{prop-changement-rep}, we have $\go{p}^p(W,B_{0},\lambda)\simeq \go{p}^p(W,\mu{\scriptscriptstyle\Box}B_{0},\frac{\lambda}{\sqrt{\mu}})$.
 
 2) Note from Lemma \ref{lemma-forme-locale} and Proposition \ref{prop-forme-globale} that the form $B_{0}$ extends uniquely to a non-degenerate invariant form $B$ on $\go{p}^p(W,B_{0},\lambda)$ such that $B(\go{g}_{i}, \go{g}_{j})=0$ if $i+j\neq 0$.
 
 3) Remember also that $\go{g}_{-1}= (\C^p[W])^*\simeq \C^p[W^*]$. Let $Q\in \C^p[W^*]$. Define a differential operator $Q(\partial)$ on $W$ by setting:
 $$Q(\partial)e^{\langle x,y\rangle }= Q(y)e^{\langle x,y\rangle }\,\,\, \text{ for all }x\in W \text{ and } y\in W^*.$$
 (Here ${\langle x,y\rangle }$ denotes the value of the linear form $y$ on $x$).
 
 Then the  isomorphism between    $\C^p[W^*]$ and $ (\C^p[W])^*$ sends $Q$ on the linear form $P\longmapsto Q(\partial)P(0)=Q(\partial)P$.

 \end{rem}
 
 \begin{prop}\label{prop-symplectic-finite} \hfill
 
 The only finite dimensional  Lie algebras of polynomial  type are:
 
 $1)$ $\go{sl}_{n+1}(\C)\simeq  \go{p}^1(\C^n, B_{0},1)$, where  $B_{0}(U,V)=2(n+1)\tr(UV)-2\tr(U)\tr(V)$ $(U,V\in \go{gl}(\C^n))$, and where the grading of $ \go{sl}(n+1,\C)$ is defined by the diagram:
 
\hskip 150pt  \vbox{\hbox{\unitlength=0.5pt
\hskip-80pt
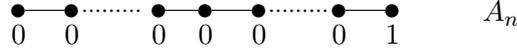
\begin{picture}(400,30)(0,10)
\put(90,10){\circle*{10}}
\put(85,-15){$0$}
\put(95,10){\line (1,0){30}}
\put(130,10){\circle*{10}}
\put(125,-15){$0$}
 \put(140,10){\circle*{1}}
\put(145,10){\circle*{1}}
\put(150,10){\circle*{1}}
\put(155,10){\circle*{1}}
\put(160,10){\circle*{1}}
\put(165,10){\circle*{1}}
\put(170,10){\circle*{1}}
\put(175,10){\circle*{1}}
\put(180,10){\circle*{1}}
 \put(195,10){\circle*{10}}
 \put(190,-15){$0$}
 \put(195,10){\line (1,0){30}}
\put(230,10){\circle*{10}}
\put(225,-15){$0$}
 
\put(235,10){\line (1,0){30}}
\put(270,10){\circle*{10}}
 \put(265,-15){$0$}
 \put(280,10){\circle*{1}}
\put(285,10){\circle*{1}}
\put(290,10){\circle*{1}}
\put(295,10){\circle*{1}}
\put(300,10){\circle*{1}}
\put(305,10){\circle*{1}}
\put(310,10){\circle*{1}}
\put(315,10){\circle*{1}}
\put(320,10){\circle*{1}}
\put(330,10){\circle*{10}}
 \put(325,-15){$0$}
\put(335,10){\line (1,0){30}}
\put(370,10){\circle*{10}}
 \put(365,-15){$1$}
 
\end{picture} 
\hskip 15pt\raisebox {-4pt}{$A_{n}$}
 
}
}
\vskip 10pt

$($see Example \ref{ex-PH-paraboliques} for the definition of the corresponding grading$)$.

$2) $ $\go{o}(2n+1,\C)\simeq  \go{p}^1(\C^n, B_{0},1)$, where $B_{0}(U,V)= 4(n-1)\tr(UV)+\frac{2}{n}\tr(U)\tr(V)$, $(U,V\in \go{gl}(\C^n))$, and where the grading  of $\go{o}(2n+1,\C)$ is defined by the diagram:
  
  \hskip 110pt\raisebox{-10pt}{{
\hbox{\unitlength=0.5pt
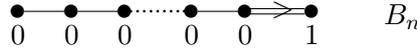
\begin{picture}(250,30) 
\put(10,10){\circle*{10}}
\put(5,-15){$0$}
\put(15,10){\line (1,0){30}}
\put(50,10){\circle*{10}}
\put(45,-15){$0$}
\put(55,10){\line (1,0){30}}
\put(90,10){\circle*{10}}
\put(85,-15){$0$}
 \put(95,10){\circle*{1}}
\put(100,10){\circle*{1}}
\put(105,10){\circle*{1}}
\put(110,10){\circle*{1}}
\put(115,10){\circle*{1}}
\put(120,10){\circle*{1}}
\put(125,10){\circle*{1}}
\put(130,10){\circle*{1}}
\put(135,10){\circle*{1}}
\put(140,10){\circle*{10}}
\put(135,-15){$0$}
\put(145,10){\line (1,0){30}}
\put(180,10){\circle*{10}}
\put(175,-15){$0$}
\put(184,12){\line (1,0){41}}
\put(184,8){\line(1,0){41}}
\put(200,5.5){$>$}
\put(230,10){\circle*{10}}
\put(225,-15){$1$}
\end{picture} \,\,\,\,\, $B_{n}$
}}}

\vskip 15pt

$3) $ $\go{sp}(n,\C)\simeq  \go{p}^2(\C^n, B_{0},1)$, where   $B_{0}(U,V)= 4(n+1)\tr(UV)-3\frac{n+1}{n}\tr(U)\tr(V)$, $(U,V\in \go{gl}(\C^n))$, and where the grading  of $\go{sp}(n,\C)$ is defined by the diagram:
  
  \hskip 110pt\raisebox{-10pt}{{
\hbox{\unitlength=0.5pt
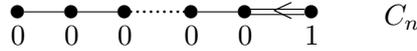
\begin{picture}(250,30) 
\put(10,10){\circle*{10}}
\put(5,-15){$0$}
\put(15,10){\line (1,0){30}}
\put(50,10){\circle*{10}}
\put(45,-15){$0$}
\put(55,10){\line (1,0){30}}
\put(90,10){\circle*{10}}
\put(85,-15){$0$}
 \put(95,10){\circle*{1}}
\put(100,10){\circle*{1}}
\put(105,10){\circle*{1}}
\put(110,10){\circle*{1}}
\put(115,10){\circle*{1}}
\put(120,10){\circle*{1}}
\put(125,10){\circle*{1}}
\put(130,10){\circle*{1}}
\put(135,10){\circle*{1}}
\put(140,10){\circle*{10}}
\put(135,-15){$0$}
\put(145,10){\line (1,0){30}}
\put(180,10){\circle*{10}}
\put(175,-15){$0$}
\put(184,12){\line (1,0){41}}
\put(184,8){\line(1,0){41}}
\put(200,5.5){$<$}
\put(230,10){\circle*{10}}
\put(225,-15){$1$}
\end{picture} \,\,\,\,\, $C_{n}$
}}}

\vskip 15pt

$4)$ $ G_{2}\simeq \go{p}^3(\C^2,B_{0},1)$  where,  up to a multiplicative constant,    $B_{0}(U,V)=24\tr(UV)- 8\tr(U)\tr(V)$ $(U,V\in \go{gl}(\C^2))$ and where the grading of $G_{2} $ is defined by the diagram:
\vskip 5 pt

\hskip 110pt\raisebox{5pt}{{
\hbox{\unitlength=0.5pt
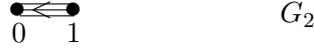
\begin{picture}(250,30)
 \put(53,10){\circle*{10}}
\put(49,-15){$0$}
\put(55,14){\line (1,0){41}}
\put(63,4){$<$}
\put(55,9,5){\line(1,0){41}}
\put(55,5){\line(1,0){41}}
\put(95,10){\circle*{10}}
\put(90,-15){$1$}
 \end{picture}$G_{2}$
}}} 

\vskip 15pt

$5)$ $ G_{2}\simeq \go{p}^1(\C^2,B_{0},1)$  where,  up to a multiplicative constant,    $B_{0}(U,V)=8\tr(UV)+8 \tr(U)\tr(V)$ $(U,V\in \go{gl}(\C^2))$ and where the grading of $G_{2} $ is defined by the diagram:
\vskip 5 pt

\hskip 110pt\raisebox{5pt}{{
\hbox{\unitlength=0.5pt
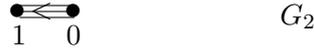
\begin{picture}(250,30)
 \put(53,10){\circle*{10}}
\put(49,-15){$1$}
\put(55,14){\line (1,0){41}}
\put(63,4){$<$}
\put(55,9,5){\line(1,0){41}}
\put(55,5){\line(1,0){41}}
\put(95,10){\circle*{10}}
\put(90,-15){$0$}
 \end{picture}$G_{2}$
}}} 

$($The isomorphisms occuring above are isomorphisms of graded quadratic graded Lie algebras, the simple algebras being endowed with the Killing form.$)$
 \end{prop}
 
 \begin{proof}\hfill
 
 As the representation $\rho_{\lambda}$ is faithful, we know from Proposition \ref{prop-finie-ss} that if $\dim(\go{p}^p(W,B_{0},\lambda))<+\infty$, then $\go{p}^p(W,B_{0},\lambda)$ is semi-simple. Hence the grading of $\dim(\go{p}^p(W,B_{0}))$ is defined by a weighted Dynkin diagram with weights equal to $0$ or $1$, see Example \ref{ex-PH-paraboliques}. As the representation is irreducible there must be only one $1$ among the weights. Remember also from Example \ref{ex-PH-paraboliques} that   $\go{g}_{0}'$ corresponds to the sub-diagram of roots of weight $0$. As in our case   $\go{g}_{0}'\simeq \go{sl}_{n}$ ($n=\dim W$), we are looking here for a connected Dynkin diagram whith $n $ vertices,  where the subdiagram of vertices of weight $0$ is of type $A_{n -1}$.  Moreover, from Definition \ref{def-type-symplectique}, the representation of $\go{sl}_{n}\simeq A_{n-1}$ on $\go{g}_{1}$ must be one of the natural representations of $\go{sl}_{n}$ on the space of polynomials of a given degree $k$ on  $\C^n$ or $(\C^n)^*$. These representations have highest weight either $k\lambda_{1}$ or $k\lambda_{n-1}$, where $\lambda_{1}$ and $\lambda_{n-1}$ are the fundamental weights corresponding to the "end roots" of the Dynkin diagram $A_{n-1}$.  This implies that the removed root is connected to one of the "end roots" of $A_{n-1}$. Hence the weighted Dynkin diagram   of $\go{p}^p(W,B_{0},\lambda)$ must be of the following type:
 
 \hskip 150pt  \vbox{\hbox{\unitlength=0.5pt
\hskip-80pt
\begin{picture}(400,30)(0,10)
\put(82,25){$\alpha_{1}$}
\put(90,10){\circle*{10}}
\put(85,-15){$0$}
\put(95,10){\line (1,0){30}}
\put(122,25){$\alpha_{2}$}
\put(130,10){\circle*{10}}
\put(125,-15){$0$}
 \put(140,10){\circle*{1}}
\put(145,10){\circle*{1}}
\put(150,10){\circle*{1}}
\put(155,10){\circle*{1}}
\put(160,10){\circle*{1}}
\put(165,10){\circle*{1}}
\put(170,10){\circle*{1}}
\put(175,10){\circle*{1}}
\put(180,10){\circle*{1}}
 \put(195,10){\circle*{10}}
 \put(190,-15){$0$}
 \put(195,10){\line (1,0){30}}
\put(230,10){\circle*{10}}
\put(225,-15){$0$}
 
\put(235,10){\line (1,0){30}}
\put(270,10){\circle*{10}}
 \put(265,-15){$0$}
 \put(280,10){\circle*{1}}
\put(285,10){\circle*{1}}
\put(290,10){\circle*{1}}
\put(295,10){\circle*{1}}
\put(300,10){\circle*{1}}
\put(305,10){\circle*{1}}
\put(310,10){\circle*{1}}
\put(315,10){\circle*{1}}
\put(320,10){\circle*{1}}
\put(315,25){$\alpha_{n-1}$}
\put(330,10){\circle*{10}}
 \put(325,-15){$0$}
\put(335,10){\line (1,0){60}}
\put(335,11){\line (1,0){60}}
\put(335,12){\line (1,0){60}}
\put(335,9){\line (1,0){60}}
\put(335,8){\line (1,0){60}}
\put(375,25){$\alpha_{n}$}
\put(390,10){\circle*{10}}
 \put(386,-15){$1$}
 
\end{picture} 
 
}
}

\vskip 20pt
where the boldface edge stands for one of the standard allowed edges (possibly multiple) in a Dynkin diagram. 

Let $\omega$ be the fundamental weight for $\go{g}_{0}'$ corresponding to the root $\alpha_{n-1}$ in the sub-diagram of type $A_{n-1}$ (picture above). Denote by $p$ the number of edges between $\alpha_{n-1}$ and $\alpha_{n}$. Then the representation $(\go{g}_{0}',\go{g}_{1}) $  has lowest weight $-p\omega$ when  $||\alpha_{n}||\geq ||\alpha_{n-1}||$    and has lowest weiht $-\omega$ when $||\alpha_{n}||\leq  ||\alpha_{n-1}||$  (for the general calculation of the lowest weight of  $(\go{g}_{0}',V)$ from the Dynkin diagram  in prehomogeneous vector spaces of parabolic type we refer to  \cite{Rubenthaler-book} p. 135).
From the list of the Dynkin diagrams we obtain exactly the five  cases in the Proposition.

  For each case we will prove that the  fundamental triplets corresponding to the graded simple algebra and to the given  Lie algebra of polynomial type are isomorphic. Then  Theorem \ref{th-bijection} implies that this isomorphism extends to a  graded isomorphism of the minimal algebras.

Let us first make some general computations. Let $\go{g}=\oplus_{i=-k}^k \go{g}_{i}$ be a simple graded Lie algebra, such that $\go{g}_{0}\simeq \go{gl}_{n}$ and such that $(\go{g}_{0}',\go{g}_{1})$ is isomorphic to $(\go{sl}_{n},\C^p[\C^n])$, in the sense  that there exists a Lie algebra isomorphism $A: \go{sl}_{n }\longrightarrow \go{g}_{0}$ and a isomorphism $\gamma: \C^p[\C^n]\longrightarrow \go{g}_{1}$ such that for $U\in \go{sl}_{n}, P\in \C^p[\C^n]$, we have $\gamma(U.P)=[A(U), \gamma(P)]$. Of course this is the case for the gradings defined by the diagrams occuring in the proposition. 

Let us extend $A$ to an isomorphism (still denoted $A$) from $\go{gl}_{n}$ to $\go{g}_{0}$ by setting $A( {\rm Id}_{n})=H$ where $H$ is the grading element. We also impose that $ {\rm Id}_{n}$ acts by the identity on $\C^p[\C^n]$. Let $K_{\go{g}}$ be the Killing form of $\go{g}$.   Define a form $B_{0}$ on $\go{gl}_{n}$ by
$B_{0}(U,V)= K_{\go{g}}(A(U),A(V))$ ($U,V\in \go{gl}_{n}$). Of course $U=U'+\lambda{\rm Id}_{n}$, $V=V'+\mu{\rm Id}_{n}$, where $U',V'\in \go{sl}_{n}$ and $\lambda =\displaystyle \frac{\tr (U)}{n}, \mu =\displaystyle \frac{\tr (V)}{n}$.

There exists a constant $c\neq 0$ such that ${K_{\go{g}}}_{|_{\go{g}_{0}'}}=cK_{\go{g}_{0}'}$. Define also $d=\sum_{i=1}^k i^2 \dim{\go{g}_{i}}$. Then:

 $$\begin{array}{rl}
B_{0}(U,V)&=K_{\go{g}}(A(U'),A(V'))+\lambda\mu K_{\go{g}}(A({\rm Id}_{n}),A({\rm Id}_{n}))\\
{}&= cK_{\go{g}_{0}'}(A(U'),A(V'))+\lambda\mu K_{\go{g}}(H,H)\\
{}&=cK_{\go{sl}_{n}}(U',V')+2\lambda\mu d\\
{}&= c2n\tr((U-\frac{\tr(U)}{n}{\rm Id}_{n})(V-\frac{\tr(V)}{n}{\rm Id}_{n}))+\frac{2d}{n^2}\tr(U)\tr(V)\\
{}&=c2n\tr(UV)+(\frac{2d}{n^2}-2c)\tr(U)\tr(V).
\end{array}$$

The pair $(A,\gamma)$ is an isomorphism between the fundamental triplets $(\go{gl}_{n}, B_{0}, (\rho_{1}, \C^p[\C^n]))$ and $(\go{g}_{0}, {K_{\go{g}}}_{|_{\go{g}_{0}}}, \go{g}_{1})$. And this isomorphism extends to an isomorphism (in fact an isomorphism of quadratic graded Lie algebras) between $\go{g}$ and $\go{p}(\C^n, B_{0}, 1)$.

It remains to compute the constants $c$ and $d$ in our cases. For the classical cases this can be done by using explicit realizations of the algebras and the expression of the Killing form by means of the trace. For the two $G_{2}$ cases one can use a root table.

The results are the following:

Case $1)$:  $K_{\go{sl}_{n+1}(\C)}=2(n+1)\tr$, $c=\frac{n+1}{n}$, $d=n$.

Case $2)$: $K_{\go{o}({2n+1},\C)}=(2n-1)\tr$, $c=2\frac{n-1}{n}$, $d=n(2n-1)$.

Case $3)$: $K_{\go{sp}({n},\C)}=(2n+2)\tr$, $c=2\frac{n+1}{n}$, $d=\frac{n(n+1)}{2}$.

Case $4)$: $c=6$, $d=8$.

Case $5)$: $c=2$, $d=24$.

This leads to the given form $B_{0}$ for the polynomial algebras.

  \end{proof}
  \vskip 5pt
  
  \begin{theorem}\label{th-symplectique-sl2} \hfill
  
  Let $\dim W> 1$. Then, except for the case $p=1$, all graded Lie algebras of polynomial type $\go{p}^p(W,B_{0},\lambda)$ are associated to an $\go{sl}_{2}$-triple.
  \end{theorem}
  
  \begin{proof}\hfill

Remember from Theorem \ref {th-CNS-sl2}, that there exists an associated    $\go{sl}_{2}$-triple if and only if there exists $X\in V\setminus \{0\}$ such that $X\notin \go{g}_{0}'.X$. We identify $W$ with $\C^n$, where $n=\dim W$.  Define $P\in \C^p[\C^n]$ by
$$P(x)=x_{1}^p+x_{2}^p+\dots+x_{n}^p.$$
By the natural representation,  $\go{sl}(n)$ acts on $ \C^p[\C^n]$ by the classical vector fields:
$$U=(a_{i,j})\longmapsto \sum_{i,j} a_{i,j} x_{i}\frac {\partial}{\partial x_{j}}.$$
If $U.P=P$, for some $U\in \go{sl}(n)$:
$$\begin{array}{rl}
 \displaystyle(\sum_{i,j} a_{i,j} x_{i}\frac {\partial}{\partial x_{j}})P(x)&=  \displaystyle(\sum_{i,j} a_{i,j} x_{i}\frac {\partial}{\partial x_{j}})(\sum_{k}x_{k}^p)\\
 = \displaystyle\sum_{i,j,k}a_{i,j}x_{i}(\frac{\partial}{\partial x_{j}}x_{k}^p)&=  \displaystyle\displaystyle\sum_{i,j,k}a_{i,j}x_{i}px_{k}^{p-1}\delta_{j,k} \\
 = \displaystyle\sum _{i,j}a_{i,j}x_{i}p x_{j}^{p-1}&= \displaystyle\sum_{k}x_{k}^p.
 \end{array}$$
 For $p>1$, this implies that $a_{i,i}=\frac{1}{p}$ for  $i=1,\dots,n$, and then $U$ cannot be in $\go{sl}(n)$. Therefore the Lie algebras of polynomial type $\go{p}^p(W,B_{0},\lambda)$ are associated to an $\go{sl}_{2}$-triple for $p>1$.
 
 Suppose now $p=1$. From Remark  \ref{rem-sl2-ind-forme} b) it is enough to prove that  for any $X\in \C^n\setminus \{0\}$, there exists $U\in \go{sl}(n)$ such that $U.X=X$. But as $\C^n\setminus \{0\}$ is a single orbit under the group $SL(n,\C)$, the map $U\longmapsto U.X$ from $\go{sl}(n)$ to $\C^n$ is surjective.

  \end{proof}

\subsection{Prehomogeneous vector spaces and dual pairs}\hfill
\vskip 5pt
In  Theorem \ref{th-symplectique-sl2} we have shown that if $n=\dim W> 1$ and $p>1$, then the polynomial $X=x_{1}^p+\dots+x_{n}^p$ is the nil-positive element of an $\go{sl}(2)$-triple associated to the algebra  $\go{p}^p(W,B_{0},\lambda)$. We will now show that, under some assumptions, if $W$ is the space of an irreducible regular  prehomogeneous vector space, and if $P$ is the corresponding fundamental invariant, then if  $p=\partial^\circ(P)$, the nil-positive element of an  associated $\go{sl}(2)$-triple of $\go{p}^p(W,B_{0},\lambda)$ can also be taken to be $P$.

\vskip 5pt
\begin{prop} \label{prop-invariant=nilpositif} \hfill

Let $W$ be a finite dimensional vector space over $\C$. Let $A\subset GL(W)$ be a connected reductive algebraic group with a one dimensional center. Denote by $\go{a}$  its Lie algebra. Suppose that $(A , W)$ is  an irreducible  regular prehomogeneous vector space. Let $P$ be the corresponding fundamental invariant.   We also make the following assumption 
$$\begin{array}{rl}{}&\go{a}=\{U\in \go{gl}(W)\,|\, \exists \mu\in \C, \, U.P=\mu P \}\\
\text{ or  equivalently }&{}\\

 {}&\go{a}'=\{U\in\go{gl}(W)\,|\, U.P=0\}  
\end{array} \eqno{(5-2-1)}$$
Then, for   any $B_{0}$,   $P$ is the nil-positive element of an associated $\go{sl}(2)$-triple in $\go{p}^{\partial^\circ(P)}(W,B_{0},\lambda)$ .
\end{prop}

\begin{proof}\hfill

 Let $H=\displaystyle \frac{1}{\lambda}\text{Id}_{W}$ be the grading element of $\go{p}^{\partial^\circ(P)}(W,B_{0},\lambda)$. Suppose that there exists $U\in \go{sl}(W)$, such that $U.P=P$. Then  we would have $(H-U).P=0$. Then from $(5-2-1)$, we obtain that $(H-U)\in \go{a}'$ and therefore $H\in \go{a}'+U \subset \go{sl}(W)$. This is not true. Then the proposition is a consequence of Theorem \ref{th-CNS-sl2}.

\end{proof}

\begin{rem}\label{rem-cond-vraie-struct}\hfill

Condition $(5-2-1)$ is always satisfied if $A$ is the structure group of a simple Jordan algebra $W$ over $\C$.    See \cite{Faraut-Koranyi-book}, Chapter VIII, exercise 5 p. 160-161.   This  case corresponds to the so-called prehomogeneous vector spaces of commutative parabolic type (see \cite{Rubenthaler-book}, ch. 5  or \cite {Muller-Rubenthaler-Schiffmann}).

 But it is also satisfied for many others prehomogeneous vector spaces.

\end{rem}

\begin{definition}\label{def-paire-duale}\hfill

Let  $\go{g}$ be Lie algebra. A pair $(\go{g}_{1},\go{g}_{2})$ of Lie subalgebras of $\go{g}$ is called a {\it dual pair} if $\go{g}_{1}$ is the centralizer of $\go{g}_{2}$ in $\go{g}$ and vice versa.
\end{definition}

 Let $P^*$ be the fundamental invariant of the dual prehomogeneous vector space $(A,W^*)$ (as $(A,W)$ is regular, the dual representation has a non trivial relative invariant). Of course $P^*$ is only defined up to a multiplicative constant.

\begin{theorem}\label{th-paire-duale}\hfill

The notations are as in Proposition \ref{prop-invariant=nilpositif}, and we suppose $\lambda\neq0$ and that condition $(5-2-1)$ is satisfied. 

Let $\widetilde{\go{b}}=Z_{\go{p}^p(W,B_{0},\lambda)}(\go{a}')$ be the centralizer of $\go{a}'$ in $\go{p}^p(W,B_{0},\lambda)$.

$1)$ One can choose $P^*$ such that $(P^*, H_{0}=\frac{2}{\lambda}\text{Id}_{W}, P)$ is an $\go{sl}(2)$-triple $($associated to the graded Lie algebra $\go{p}^p(W,B_{0},\lambda)$$)$. Then, if $\go{b} $ is the Lie subalgebra of $\go{p}^p(W,B_{0},\lambda)$ isomorphic to $\go{sl}(2) $ generated by this triple, we have $$\go{b}=\widetilde{\go{b}}\cap \Gamma(\go{p}^p(W,B_{0},\lambda)) $$

$2)$ Moreover $(\go{a}', \widetilde{\go{b}}) $ is a dual pair in $\go{p}^p(W,B_{0},\lambda)$.

 \end{theorem}
 
 \begin{proof}\hfill
 
 1) As $\widetilde{\go{b}}=Z_{\go{p}^p(W,B_{0},\lambda)}(\go{a}')$ and as $\go{a}'\subset \go{gl}(W)=\go{g}_{0}$, we see that $\widetilde{\go{b}} $ is a graded subalgebra of $\go{p}^p(W,B_{0},\lambda)$.
 
 Hence 
 $$ \widetilde{\go{b}}\cap \Gamma(\go{p}^p(W,B_{0},\lambda)) =  \widetilde{\go{b}}\cap \go{g}_{-1}\oplus  \widetilde{\go{b}}\cap \go{g}_{0}\oplus  \widetilde{\go{b}}\cap \go{g}_{1}.$$
 But as $W$ is irreducible under $\go{a}$, we obtain that 
$$\widetilde{\go{b}}\cap \go{g}_{1}=\C. P \text{ and }\widetilde{\go{b}}\cap \go{g}_{-1}=\C. P^*.$$
From the Schur Lemma we get also
$$\widetilde{\go{b}}\cap \go{g}_{0}=\{U\in \go{gl}(W), [U,\go{a}']=0\}=\C. \text{Id}_{W}.$$
As  $[P^*,P]\in \widetilde{\go{b}}\cap \go{g}_{0}$, we have $[P^*,P]=\gamma \text{Id}_{W}, \, \gamma\in \C.$ 

Suppose that $[P^*,P]=0$. Let $B$ be the non-degenerate invariant bilinear form which extends $B$ (see section 3.4). Then
$$0=B([P^*,P], \text{Id}_{W})=B(P^*,[P,\text{Id}_{W}])=-\lambda B(P^*,P)=-\lambda P^*(\partial)P(0)$$
But it is well known from the theory of prehomogeneous space that $P^*(\partial)P(0)\neq0$ (see for example the computation on p. 19 in \cite{Rubenthaler-book}). Hence $(\frac{2}{\lambda \gamma}P^*,\frac{2}{\lambda}\text{Id}_{W}, P) $ is an $\go{sl}(2)$-triple and $\go{b}=\widetilde{\go{b}}\cap \Gamma(\go{p}^p(W,B_{0},\lambda)) $.

2) For the second assertion, we have just to prove that $Z_{\go{p}^p(W,B_{0},\lambda)}( \widetilde{\go{b}})\subset \go{a}'$. As $\go{b}=\widetilde{\go{b}}\cap \Gamma(\go{p}^p(W,B_{0},\lambda)) $, we have 
$$Z_{\go{p}^p(W,B_{0},\lambda)}( \widetilde{\go{b}})\subset Z_{\go{p}^p(W,B_{0},\lambda)}({\go{b}})\subset Z_{\go{p}^p(W,B_{0},\lambda)}(\frac{2}{\lambda}\text{Id}_{W})= \go{gl}(W).$$

Therefore $Z_{\go{p}^p(W,B_{0},\lambda)}(\widetilde{\go{b}})\subset Z_{\go{gl}(W)}(P)=\go{a}'$ (condition $(5-2-1)$).

 \end{proof}
 
 \begin{example}\label{ex-paireduale-O(n)-SL(2)} {\bf The dual pair $(\go{o}(n), \go{sl}(2))$}\hfill
 
Let $O(n)$ be the orthogonal group over $\C$ of size $n$ and let $\go{o}(n)$ be its Lie algebra. Define $A=\C^*\times O(n)$ and $W=\C^n$. The natural representation $(A,W)$ is an irreducible regular prehomogeneous vector space whose fundamental relative invariant can be chosen to be  the quadratic form $P(x)= x_{1}^2+\dots+x_{n}^2 $. Condition $(5-2-1)$ is satisfied as $A$ is the structure group of the  Jordan algebra $\C^n$ (cf. Remark \ref{rem-cond-vraie-struct}). Consider the trace form on $\go{gl}(n)$ defined by $B_{0}(U,V)=\tr(UV)$. Then, in the preceding notations, the algebra $\go{p}^2(W, tr, -2)$ is the ordinary symplectic algebra $\go{p}(n,\C)$ with the grading defined as follows (here $S_{n}(\C)$ stands for the $n\times n$ symmetric matrices):
$$\begin{array}{rcl}
V^*\simeq \C^2[(\C^n)^*]\simeq S_{n}(\C)\simeq {\go g}_{-1}&=& \{\begin{pmatrix}0&0\\
 Y&0
\end{pmatrix}, Y\in S_{n}(\C)\}\\

\go{gl}(n)\simeq {\go g}_{0}&=& \{\begin{pmatrix}A&0\\
 0&-^t A
\end{pmatrix}, A\in \go{gl}(n) \}\\
  
V\simeq \C^2[\C^n]\simeq S_{n}(\C \simeq {\go g}_{1}&=& \{\begin{pmatrix}0&X\\
 0&0
\end{pmatrix}, X\in S_{n}(\C)\}\\
 \end{array}$$
 
 Here the quadratic form $P$ can be identified with the matrix $\begin{pmatrix}0&\text{Id}_{n}\\
 0&0
\end{pmatrix}\in \go{g}_{1}$. In this case, keeping the preceding notations, we obtain the dual pair $(\go{a}', \widetilde{\go{b}})$ where $\go{a}'= \{\begin{pmatrix}A&0\\
 0& A
\end{pmatrix}, A\in \go{o}(n) \}\simeq \go{o}(n)$, and where $\widetilde{\go{b}}=\go{b}=\{\begin{pmatrix}a\text{Id}_{n}&b\text{Id}_{n}\\
 c\text{Id}_{n}& -a\text{Id}_{n}
\end{pmatrix}, a,b,c\in \C \}$.  This is the archetype of a dual pair   in $\go{p}(n, \C)$ (see \cite{Howe}, p. 556).

Through our construction this pair appears to be associated to the prehomogeneous vector space $(O(n)\times \C^*,\C^n)$.
 
 \end{example}
 
 \begin{rem} In the notations of the preceding example,  we could  now take  $(\go{g}_{0},\go{g}_{1})$  as the starting prehomogeneous space $(\go{a},W)$, and  do the same construction as in Theorem \ref{th-paire-duale}. But then, as the degree of the fundamental relative invariant (the determinant of the symmetric matrices)  is of degree $n$, we are led to the algebras $\go{p}^n( \go{gl}(\frac{n(n+1)}{2}), B_{0}, \lambda)$ , which are of infinite dimension, according to Proposition \ref{prop-symplectic-finite}.
 
  \end{rem}



 \end{document}